\documentclass[a4paper, 11pt]{amsart}   
\usepackage{mathptmx, amssymb,amscd,latexsym, eulervm}   
\usepackage{amsmath}
\usepackage{amsthm}
\usepackage{mathdots}
\usepackage[pagebackref, colorlinks=true, linkcolor=blue, citecolor=red]{hyperref}
\usepackage{setspace}
\usepackage{tabularx}
\usepackage{amsfonts}
\usepackage{epigraph}
\usepackage{paralist}
\usepackage{aliascnt}
\usepackage{amscd}
\usepackage{blkarray}
\usepackage{booktabs}
\usepackage{enumitem}
\usepackage{cite}
\usepackage[normalem]{ulem}
\usepackage{setspace}
\usepackage[inner=2.4cm,outer=2.4cm, top = 3cm, bottom=2.7cm, marginparwidth=2cm]{geometry}
\usepackage[capitalize,nameinlink]{cleveref}
\usepackage{tikz, tikz-cd}
\usepackage{calligra,mathrsfs}
\usepackage{todonotes}

\usetikzlibrary{matrix}
\usetikzlibrary{arrows,calc}
\allowdisplaybreaks

\newtheorem{thm}{Theorem}[section]
\newtheorem{lemma}[thm]{Lemma}
\newtheorem{cor}[thm]{Corollary}
\newtheorem{prop}[thm]{Proposition}

\theoremstyle{definition}
\newtheorem{example}[thm]{Example}

\newtheorem{remark}[thm]{Remark}
\newtheorem{definition}[thm]{Definition}
\newtheorem{definitions}[thm]{Definitions}
\newtheorem{conjecture}[thm]{Conjecture}
\newtheorem{question}[thm]{Question}

\numberwithin{equation}{section}
\numberwithin{figure}{section}

\newcommand{\Spec}{\mathrm{Spec}}

\newcommand{\Hilb}{\mathrm{Hilb}}
\newcommand{\Def}{\mathrm{Def}}
\newcommand{\Hom}{\mathrm{Hom}}
\newcommand{\Ext}{\mathrm{Ext}}
\newcommand{\Tor}{\mathrm{Tor}}
\newcommand{\Gr}{\mathrm{Gr}}

\newcommand{\GL}{\mathrm{GL}}

\newcommand{\soc}{\mathrm{soc}}
\newcommand{\Ann}{\mathrm{Ann}}

\renewcommand{\AA}{\mathbb{A}}

\newcommand{\N}{\mathbb{N}}
\newcommand{\Z}{\mathbb{Z}}

\newcommand{\kk}{{\Bbbk}}
\newcommand{\mm}{{\mathfrak{m}}}

\newcommand{\bfa}{\mathbf{a}}
\newcommand{\bfb}{\mathbf{b}}
\newcommand{\bfc}{\mathbf{c}}
\newcommand{\bfd}{\mathbf{d}}
\newcommand{\bfe}{\mathbf{e}}

\newcommand{\bfg}{\mathbf{g}}
\newcommand{\bfp}{\mathbf{p}}
\newcommand{\bfq}{\mathbf{q}}
\newcommand{\bfv}{\mathbf{v}}
\newcommand{\bfw}{\mathbf{w}}

\newcommand{\bfs}{\mathbf{s}}
\newcommand{\bft}{\mathbf{t}}

\newcommand{\ua}{{\underline{\alpha}}}
\newcommand{\ub}{{\underline{\beta}}}

\newcommand{\mcO}{{\mathcal{O}}}

\newcommand{\msI}{{\mathscr{I}}}

\newcommand{\msO}{{\mathscr{O}}}

\newcommand{\n}{\mathrm{n}}

\newcommand{\p}{\mathrm{p}}

\newcommand{\ppn}{\mathrm{ppn}}
\newcommand{\pnp}{\mathrm{pnp}}
\newcommand{\npp}{\mathrm{npp}}
\newcommand{\nnp}{\mathrm{nnp}}
\newcommand{\npn}{\mathrm{npn}}
\newcommand{\pnn}{\mathrm{pnn}}

\newcommand{\Pf}{\mathrm{Pf}}
\newcommand{\gen}{\mathrm{gen}}

\newcommand{\tri}{\mathrm{tri}}

\DeclareMathOperator{\lci}{lci}%
\DeclareMathOperator{\Sym}{Sym}%
\newcommand{\sym}{\operatorname{sym}}
\newcommand{\sm}{\operatorname{sm}}
\DeclareMathOperator{\im}{im}%
\DeclareMathOperator{\coker}{coker}%

\newcommand{\nestedHilbX}[2]{\Hilb^{(#1, #2)}(X)}
\newcommand{\nestedlcilocus}[2]{\Hilb^{(#1, #2)}_{\star, \lci}(X)}
\newcommand{\mcZ}{\mathcal{Z}}
\newcommand{\onto}{\twoheadrightarrow}
\newcommand{\powerseries}{\kk[\![t]\!]}

\newcommand{\cK}{\mathcal{K}}%

\makeatletter
\@namedef{subjclassname@2020}{
  \textup{2020} Mathematics Subject Classification}
\makeatother

\newcommand{\monomialIdealPicture}[1]{
    \begin{tikzpicture}[x=(220:1cm), y=(-40:1cm), z=(90:0.707cm),scale=0.5]
        \foreach \m [count=\y] in {#1}{
          \foreach \n [count=\x] in \m {
          \ifnum \n>0
              \foreach \z in {1,...,\n}{
                \draw [fill=orange!30] (\x+1,\y,\z) -- (\x+1,\y+1,\z) -- (\x+1, \y+1, \z-1) -- (\x+1, \y, \z-1) -- cycle;
                \draw [fill=orange!40] (\x,\y+1,\z) -- (\x+1,\y+1,\z) -- (\x+1, \y+1, \z-1) -- (\x, \y+1, \z-1) -- cycle;
                \draw [fill=orange!10] (\x,\y,\z)   -- (\x+1,\y,\z)   -- (\x+1, \y+1, \z)   -- (\x, \y+1, \z) -- cycle;  
              }
             \fi
          }
        }
    \end{tikzpicture}
}
\newcommand{\monomialIdealPictureSmall}[1]{
    \begin{tikzpicture}[x=(220:1cm), y=(-40:1cm), z=(90:0.707cm),scale=0.35]
        \foreach \m [count=\y] in {#1}{
          \foreach \n [count=\x] in \m {
          \ifnum \n>0
              \foreach \z in {1,...,\n}{
                \draw [fill=orange!30] (\x+1,\y,\z) -- (\x+1,\y+1,\z) -- (\x+1, \y+1, \z-1) -- (\x+1, \y, \z-1) -- cycle;
                \draw [fill=orange!40] (\x,\y+1,\z) -- (\x+1,\y+1,\z) -- (\x+1, \y+1, \z-1) -- (\x, \y+1, \z-1) -- cycle;
                \draw [fill=orange!10] (\x,\y,\z)   -- (\x+1,\y,\z)   -- (\x+1, \y+1, \z)   -- (\x, \y+1, \z) -- cycle;  
              }
             \fi
          }
        }
    \end{tikzpicture}
}

\newcommand{\degold}{\dim_{\kk}}%
\begin{document}

\author[J.\,Jelisiejew, R.\,Ramkumar, A. \,Sammartano]{Joachim~Jelisiejew, Ritvik~Ramkumar and Alessio~Sammartano}
\address{(Joachim Jelisiejew)
    Faculty of Mathematics, Informatics, and Mechanics,
    University of Warsaw, Banacha 2, 02-097, Warsaw, Poland}
\email{j.jelisiejew@uw.edu.pl}
\address{(Ritvik Ramkumar) Department of Mathematics\\University of Notre Dame\\South Bend, IN\\USA}
\email{rramkuma@nd.edu}
\address{(Alessio Sammartano) Dipartimento di Matematica \\ Politecnico di Milano \\ Milan \\ Italy}
\email{alessio.sammartano@polimi.it}

\title{The Hilbert scheme of points on a threefold:  broken
Gorenstein structures and linkage}

\begin{abstract}
We investigate the Hilbert scheme of points on a smooth threefold.
We introduce a notion of broken Gorenstein structure for finite schemes, 
and show that its existence guarantees smoothness on the Hilbert scheme. 
Moreover, we conjecture that it is exhaustive: every smooth  point admits a broken Gorenstein structure. 
We give an explicit characterization of the smooth points on the Hilbert scheme of $\AA^3$ corresponding to monomial ideals.
We investigate the nature of the singular points, 
and prove several conjectures by Hu. 
Along the way, we obtain a number of additional  results,
related to linkage classes, 
nested Hilbert schemes, and a bundle on the Hilbert scheme of a surface.
\end{abstract}
\maketitle

\section{Introduction}

The Hilbert scheme of $d$ points on a smooth variety $X$, denoted by
    $\Hilb^d(X)$, parametrizes zero-dimensional subschemes of $X$ of
    degree $d$. When $X$ is a smooth
    surface, the Hilbert scheme $\Hilb^d(X)$ is irreducible and smooth
    \cite{fogarty}.  This result has laid the groundwork for important applications across numerous fields:
combinatorics~\cite{Haiman_macdonald, haiman_factorial_conjecture},
enumerative geometry \cite{nakajima_lectures_on_Hilbert_schemes, Grojnowski,
Nakajima__instanton_counting, Goetsche}, moduli spaces of sheaves
\cite{Huybrechts_Lehn, MNOP, MNOPTwo}, 
 topology and $K$-theory
\cite{Ellingsrud_Stromme__On_the_homology, Ellingsrud_Stromme_two}, 
and knot
theory~\cite{Oblomkov_Shende, Oblomkov_Rozansky, GNR}.
However, when $\dim(X) \geq 4$, the Hilbert scheme has generically
    non-reduced components and is expected to exhibit extreme pathological
    behavior~\cite{Jelisiejew__Pathologies, Jelisiejew_Fourfolds}.

When $X$ is a smooth threefold, there is an interesting mixture of irregularities and structure, though very few results are known.  
Much of the effort has  focused  on its tangent spaces, 
	particularly its maximal dimension \cite{Briancon_Iarrobino__dimension_of_punctual_Hilbert_scheme, Sturmfels__counterexamples, RS22, Rezaee__Conjectural_most_singular, mackenzie2025proof}
and the parity conjecture \cite{MNOP,
    Ramkumar_Sammartano_parity, Graffeo_Giovenzana_Lella}. 
The superpotential   description~\cite{Behrend} restricts the possible singularities.
The most interesting component is   the smoothable component $\Hilb^{d, \sm}(X)$, which parametrizes  tuples of $d$ points.
This is the only component of  $\Hilb^d(X)$ if $d \leq 11$,
but not for $d\geq 78$ \cite{iarrobino_reducibility, iarrobino_compressed_artin}. 
Moreover,   the smoothable component is quite special as it is conjectured to be
normal and Cohen-Macaulay
\cite[Conjecture~5.2.1]{Haiman_macdonald}
  and expected to be 
   the only generically reduced component.  
It is smooth if $d\leq
    3$,  
    and its singularities are   known only for  $4 \leq d \leq 6$ 
    \cite{Katz__4points, H23, Hauenstein_Manivel_Szendroi__Katz_bis}.

\begin{question}\label{que:smooth}
Let $X$ be a smooth threefold. What are the smooth points of $\Hilb^d(X)$ that lie on the smoothable component?
\end{question}

Since this question is local on $X$, we may assume that $X = \AA^3$. 
The known classes of smooth points are all classical; they emerge from structure theorems for free resolutions. These include complete and almost complete intersections, algebras with embedding dimension two and Gorenstein algebras, see \cite[\S7-\S8]{hartshorne_deformation_theory}. Additionally, one can form the link of any of these subschemes along a complete intersection, resulting in what are known as licci ideals, which also correspond to smooth points (see \Cref{ssec:linkage}). The classes arising from structure theorems do not completely cover the smooth locus, while we conjecture that licci ideals do (\Cref{main_conjecture}). 
However,
it is challenging to systematically produce licci elements within a given $\Hilb^d(\mathbb{A}^3)$.

One of the main goals of the current article is to propose an explicit answer to \Cref{que:smooth}.
To this end, we introduce a new concept, which we call \emph{broken Gorenstein structures}, and show that their existence ensures smoothness on the Hilbert scheme.
Moreover, we conjecture that the converse statement also holds, and we verify this conjecture in the case of monomial schemes.

We now turn to a detailed description of the results presented in this paper.

\subsection{Monomial ideals}
To   elucidate our framework, we start with the points corresponding to monomial subschemes in $\Hilb^d(\AA^3)$. 
These points lie on the smoothable component \cite{CEVV}. 
They play a key role in enumerative geometry~\cite{Behrend__Fantechi} and are essential in torus actions and Bia{\l}ynicki-Birula cells \cite{Ellingsrud_Stromme__On_the_homology}.
They are also significant in combinatorics, where their count is governed by the well-known MacMahon formula, which is continuously being refined~\cite{MacMahon1, MacMahon2}.

Let $S = \kk[x, y, z]$ and, by a slight abuse of notation, 
we identify monomials of $S$ with their exponent vectors in $\mathbb{N}^3$ and will freely
interchange between the two. 
For a monomial ideal $I \subseteq S$, we define
its \textbf{staircase} $E_I \subseteq \mathbb{N}^3$ to be the monomials of $S$
that are not in $I$. 
The \textbf{socle} $\soc(S/I)$ is spanned by the maximal elements of $E_I$ with respect to the usual partial order.
For example, here are the staircase diagrams for $I_1 = (x^2, xy, xz, y^2, yz, z^3)$ and $I_2 = (x^2, xy, xz, y^2, z^2)$, respectively:
\begin{equation}\label{eq:monomialExample}
    \underset{\displaystyle{E_{I_1} = \{1, \underline{x}, \underline{y}, z, \underline{z^2}\}}}{\monomialIdealPicture{{3,1},{1,0}}}
\qquad \qquad \qquad \qquad \qquad
\underset{\displaystyle{E_{I_2} = \{1, \underline{x}, {y}, z, \underline{yz}\}}}{\monomialIdealPicture{{2, 1},
{2, 0}}}
\end{equation}
The underlined elements of $E_{I_1}$ and $E_{I_2}$  are in the socle.

The following definition
encapsulates our 
main combinatorial insight.
\begin{definition}
Let $I \subseteq S$ be a  cofinite monomial ideal.
A \textbf{singularizing triple} for $I$ is a triple of monomials 
$\{\bfa, \bfb, \bfc\} \subseteq \soc(S/I)$ 
such that
\[
\bfa_1 >\bfb_1, \bfc_1,
\quad
\bfb_2 > \bfa_2, \bfc_2,
\quad
\bfc_3 > \bfa_3, \bfb_3.
\]

\end{definition}

\vspace*{-0.2cm}

This notion allows us to formulate a  classification of smooth monomial points in the Hilbert scheme.

\begin{thm}[\Cref{ThmSmoothMonomialClassification}] \label{ThmSmoothMonomialIntro}
    Let $I \subseteq S = \kk[x, y, z]$ be a monomial ideal and $[S/I] \in \Hilb^d(\AA^3)$ the corresponding point. The following conditions are equivalent
\begin{enumerate}
    \item\label{monomialItem1} The point $[S/I]$ is a smooth point of the Hilbert scheme.
\item\label{monomialItem2} The ideal $I$ admits no singularizing triple.
\item\label{monomialItem3} The ideal $I$ is licci.
\end{enumerate}
\end{thm}
\vspace*{-0.2cm}

A remarkable feature of this criterion is that smoothness is detected  directly  from the dual generators (cf. \Cref{ssec:macInv}) of $S/I$.
In general, one cannot expect to extract deformation-theoretic information simply by inspecting 
 generators or  dual generators,
except in some very special cases such as complete intersections or ideals with small type and small deviation
\cite[Theorem 6.2]{GNW__ADE}.

For instance, in~\eqref{eq:monomialExample}, the first ideal corresponds to a singular point, with a singularizing triple $  \{x, y, z^2\}$,  while the second ideal is a smooth point, since $\soc(S/I_2)$ contains only two monomials.
For a more complicated example,
of the following two staircases 
\begin{equation*}
    \monomialIdealPictureSmall{{7, 7, 7, 7, 7, 7, 7}, {7, 6, 6, 6, 6, 6, 1}, {7, 6, 5, 5, 5, 2, 1}, {7, 6,
       5, 4, 3, 2, 1}, {7, 6, 5, 3, 3, 2, 1}, {7, 6, 2, 2, 2, 2, 1}, {7, 1, 1, 1, 1,
       1, 1}}\hspace*{2cm}
       \monomialIdealPictureSmall{{7, 7, 7, 7, 7, 7, 1}, {6, 6, 6, 6, 6, 2, 1}, {6, 5, 5, 5, 3, 2, 1}, {6, 5,
       4, 3, 3, 2, 1}, {6, 5, 4, 2, 2, 2, 1}, {6, 5, 1, 1, 1, 1, 1}, {6}}
\end{equation*}

\vspace*{-0.3cm}
the first one gives a smooth point, 
whereas the second one gives a singular point.

After completing the first version of this paper, 
we became aware of the preprint
\cite{Huibregtse_monomial},
whose  main result, 
\cite[Theorem 10.3.1]{Huibregtse_monomial},
also provides a classification of smooth monomial points of $\Hilb^d(\AA^3)$,
 in terms of ``compound boxes'',  a certain recursive decomposition   of the staircase.
In fact, 
this result corresponds to 
the implication $\eqref{monomialItem1} \Leftrightarrow \eqref{monomialItem2}$ in \Cref{ThmSmoothMonomialIntro},
see  \Cref{RemHuibregtse}.

\subsection{Broken Gorenstein structures} \label{broken_intro}
We introduce the following  new notion.
\begin{definition}[Broken Gorenstein algebras]\label{def:brokenGorenstein}
Let $R$ be a finite $\kk$-algebra.
We say that $R$ has a \textbf{$0$-broken Gorenstein structure} if  $R$ is Gorenstein.
For    $k\geq 1$, a \textbf{$k$-broken Gorenstein structure} on $R$ consists of a
    short exact sequence of $R$-modules
$
        0\to \cK\to R\to R_0\to 0
 $
such that:
\begin{enumerate}
\item the algebra $R/\Ann(\cK)$ has a $(k-1)$-broken Gorenstein structure,
\item the algebra  $R_0$ is Gorenstein, and
\item the $R$-module $\cK$ is either cyclic or cocyclic.
\end{enumerate}
Recall that a finite $R$-module $M$ is {\bf cocyclic} if the dual $M^{\vee} = \Hom_\kk(M,\kk)$, with its natural $R$-module structure, is a cyclic $R$-module (see \Cref{ssec:zerodimensional}). 

A \textbf{broken Gorenstein structure} on $R$ is a $k$-broken  Gorenstein structure for some $k$.
A \textbf{broken Gorenstein structure without flips} is defined inductively in the
same way, but with the additional requirement that $\cK$ is always cyclic.
Finally, if $R$ admits a broken Gorenstein structure, we call it a
\textbf{broken Gorenstein algebra}.
\end{definition}

Here is our main result regarding broken Gorenstein algebras which are
quotients of $S = \kk[x, y, z]$.

\begin{thm}[{\Cref{ref:smoothnessOfSmoothableBroken:cor}}]\label{ref:introSmoothness:thm}
If $R = S/I$ is a smoothable finite algebra with a broken Gorenstein structure, then the corresponding point $[R] \in \Hilb(\mathbb{A}^3)$ is smooth.
\end{thm}

Since the definition of  a broken Gorenstein algebra is
rather involved, we  illustrate it with a couple of examples. 
 We begin with the case of a $k$-broken Gorenstein structure without flips. 
In this case, $\cK\subseteq R $ is a principal ideal, say $\cK = \alpha_1 R$.
From the $(k-1)$-broken Gorenstein structure on $\cK \simeq R/\Ann(\cK)$
we obtain an exact sequence
$
0\to \cK' \to \cK \to R_1 \to 0
$
such
that $\cK'\subseteq \cK \subseteq R$ is also a principal ideal,  say $\cK' = \alpha_1\alpha_2 R$.
Continuing in this fashion, we see that a $k$-broken Gorenstein structure
without flips is equivalent to a flag of principal ideals
\begin{equation}\label{eq:noflipsFlag}
    0 \subsetneq \alpha_k\alpha_{k-1} \cdots \alpha_1 R\subsetneq  \ldots
    \subsetneq \alpha_1\alpha_2R \subsetneq \alpha_1 R \subsetneq R,
\end{equation}
where the subquotients $(\alpha_{i} \ldots \alpha_1) R/
(\alpha_{i+1}\alpha_{i} \ldots \alpha_1) R$
correspond to
Gorenstein algebras (where $\alpha_0 := 1$, $\alpha_{k+1} := 0$).
The class of algebras that admit such a flag is
large:
in addition to  Gorenstein algebras, 
it includes   algebras $R =\kk[x, y]/I$,
and monomial algebras $R = \kk[x, y, z]/I$ 
 such that $[R]$ is smooth.
Informally,
the structure encodes the ``broken up'' Gorenstein
subquotients, hence the name.

For example, for the algebra $R = S/I_2$ with $I_2$ as in \eqref{eq:monomialExample},
the inclusions $0 \subseteq xR \subseteq R$ give a $1$-broken Gorenstein structure. Here is an example of an algebra that does not admit a broken Gorenstein structure.

\begin{example}
\label{ex} 
Let $R = \kk[x, y, z]/(x, y, z)^2$. Up to a change of coordinates, the only surjections $R \onto R_0$ with $R_0$ being Gorenstein are
    $R \onto \kk$ and $R \onto \kk[x]/(x^2)$. 
    The kernels of these surjections, $(x, y, z)R$ and $(y, z)R$, are neither cyclic nor cocyclic.
    Thus, $R$ does not admit a broken Gorenstein  structure.
\end{example}

The treatment of broken Gorenstein structures with flips is more intricate (see  \Cref{sec:BGS}).
We now describe an effective method for constructing
broken Gorenstein algebras using Macaulay's inverse systems.

\begin{example}[\Cref{exampleFlip}]\label{example_two_step}

Let $S = \kk[x, y, z]$ and let $P = \kk[X, Y, Z]$ be another polynomial ring, 
viewed as  an
$S$-module via the contraction action
 (see \Cref{ssec:macInv} for details).
Let $f\in P$ and $g\in \kk[Y, Z]\subseteq P$ be polynomials. 
The algebra $R = S/\Ann(f, g)$ admits a  broken Gorenstein structure.
 Taking $f = X$, $g = YZ$, we
    recover the second monomial ideal in Example ~\eqref{eq:monomialExample}. See also \Cref{example_flips}.
\end{example}

In \Cref{ref:smoothnessOfSmoothableBroken:cor} we also obtain two
related results: that $[R \onto R_0]$ is a smooth point of the smoothable component
of the nested Hilbert scheme and that all infinitesimal deformations of $\cK$
can be embedded in deformations of $R$. Both of these facts are quite surprising, even with the prior assumption that $[R]$ is smooth, hinting at the potential for further structure to be uncovered.

We conjecture that broken Gorenstein structures capture all the smooth points of $\Hilb^d(\mathbb{A}^3)$ that lie on the smoothable component. 

\begin{conjecture} \label{main_conjecture}
    Let $S = \kk[x,y,z]$ and $[S/I] \in \Hilb^d(\AA^3)$ be a smoothable point. Then the following are equivalent
    \begin{enumerate}
        \item\label{it:main_conj:brokenGor} The algebra $S/I$ admits a broken Gorenstein structure.
        \item\label{it:main_conj:smooth} The point $[S/I]$ is smooth on the Hilbert scheme.
        \item\label{it:main_conj:licci} The ideal $I$ is licci.
    \end{enumerate}
\end{conjecture}

The implication $\eqref{it:main_conj:licci}\Rightarrow\eqref{it:main_conj:smooth}$ is
classical
\cite[6.4.4]{BuchweitzThesis}, \cite[Exercise 18.7]{hartshorne_deformation_theory}.
The implication
$\eqref{it:main_conj:brokenGor}\Rightarrow\eqref{it:main_conj:smooth}$ is
\Cref{ref:introSmoothness:thm}.

We establish this conjecture when $I$ is a monomial ideal.
Specifically, the absence of a singularizing triple is equivalent to having a
broken Gorenstein structure, which, in fact, can be chosen to have no flips
(\Cref{ThmSmoothMonomialClassification}). 
Furthermore, the conjecture holds
for $d \leq 6$, by a check with Poonen's list~\cite{Poonen},
see \Cref{RemPoonen}.

A crucial ingredient of our proof of \Cref{ref:smoothnessOfSmoothableBroken:cor} is the bicanonical module $\Sym^2_R \omega_R$ of an algebra $R$ (see \Cref{ssec:bicanonical}). 
Surprisingly, the bicanonical module manifests itself naturally in a variety of situations. For instance, we show that it yields an interesting torus-equivariant rank $d$ bundle on the Hilbert scheme $\Hilb^d(\mathbb{A}^2)$, and observe that it is connected to the \emph{Bodensee programme} of \cite[p.~9, arxiv version]{Michalek__Bodensee} through \Cref{ref:compatibility:lem}. A more comprehensive study of the bicanonical module will be forthcoming.

It is natural to ask which results about Gorenstein algebras generalize to broken Gorenstein algebras.
We take a step in this direction by presenting a Pfaffian description of
$I\subseteq S = \kk[x,y,z]$, whenever $S/I$ is equipped with a broken Gorenstein algebra without flips. 
In fact, we do this more generally for ideals of codimension three in a regular local ring $S$;
 the notion of broken Gorenstein without flips (\Cref{def:brokenGorenstein}) extends to this setup.

For a skew-symmetric matrix $A$, let $\Pf(A)_i$  denote the Pfaffian of the submatrix obtained by removing the $i$-th row and $i$-th column of $A$. Similarly, we define $\Pf(A)_{\geq 2}$ to be the ideal $(\Pf(A)_2,\Pf(A)_3, \ldots)$ and $\Pf(A)$ to be the ideal $(\Pf(A)_1, \Pf(A)_2, \ldots) = (\Pf(A)_1) + \Pf(A)_{\geq 2}$.

\begin{thm}[Structure theorem for broken Gorenstein structures without flips]\label{theorem_pfaffian_structure} 
Let $S$ be a regular local ring and $I \subseteq S$ be an ideal. Assume that $R = S/I$ has a
$k$-broken Gorenstein structure without flips whose subquotients have
codimension three.
Let $\alpha_1, \ldots ,\alpha_k$ be any lifts to $S$ of the elements defined
in~\eqref{eq:noflipsFlag}.
Then there exist $k+1$ skew-symmetric matrices $A_0,\dots,A_{k}$ with entries
in $S$ such that $\Pf(A_i)$ defines the codimension three Gorenstein
quotient $\alpha_i\cdots\alpha_1R/\alpha_{i+1}\cdots\alpha_1R$, that $\Pf(A_i)_1 = \alpha_{i+1}$ for $i=0, \ldots
,k-1$, and additionally
\begin{equation}\label{eq:brokenExplicitGens}
    I = \Pf(A_0)_{\geq 2} + \alpha_1\Pf(A_1)_{\geq 2} +
    \alpha_1\alpha_2\Pf(A_2)_{\geq 2} + \cdots + \alpha_1 \cdots
    \alpha_{k-1}\Pf(A_{k-1})_{\geq 2} + \alpha_1 \cdots \alpha_k\Pf(A_k).
\end{equation}
In particular, the ideal $I$ is determined by $A_0, \ldots ,A_k$ alone.
\end{thm}

This theorem provides  a common generalization of the Buchsbaum-Eisenbud and Hilbert-Burch theorems.

The matrices $A_0,\dots,A_{k}$ in  \Cref{theorem_pfaffian_structure} satisfy the following relation for all  $i = 0, \ldots, k-1$:

\begin{align} \label{pfaffian_condition}
\left(\sum_{j=0}^{i} \alpha_0\cdots\alpha_j\Pf(A_j)_{\geq 2} \right)\cap (\alpha_0\cdots\alpha_{i})	\,
\subseteq
\,
  \alpha_0\cdots\alpha_{i}\Pf(A_{i})_{\geq 2}+ (\alpha_0\cdots\alpha_{i+1}).
\end{align}

The following converse of \Cref{theorem_pfaffian_structure} holds:

\begin{prop}\label{Prop_pfaffian_converse}
 Let $A_0, \ldots ,A_k$ be skew-symmetric matrices with entries in a regular local ring $S$.
Assume that $\Pf(A_i)$ is a  Gorenstein ideal of codimension three for each $i$, 
and assume that  \Cref{pfaffian_condition} holds.
Then, formula~\eqref{eq:brokenExplicitGens} defines the ideal of a $k$-broken
Gorenstein algebra without flips.
\end{prop}

\subsection{Singular points}
As with the smooth locus, little is known classically about the
possible singularities of $\Hilb^d(\mathbb{A}^3)$. Let $S = \kk[x,y,z]$.
For a point $[S/I] \in \Hilb^d(\AA^3)$, let $T(I) := T_{[S/I]}\Hilb^d(\AA^3)$ denote the tangent space.
Recently, Hu~\cite[\S4.5]{H23}, motivated by conjectures on the Euler characteristic of certain tautological bundles on the Hilbert scheme, formulated an inspiring set of conjectures for the singularities of $\Hilb^d(\mathbb{A}^3)$. Two of the conjectures, which he has verified for $d \leq 7$, are as follows:

\begin{conjecture}[{\cite[4.25]{H23}}]\label{conj:hu1}
If $[S/I] \in \Hilb^d(\AA^3)$ is smoothable and singular, then $\dim_{\kk} T(I) \geq 3d+6$.
\end{conjecture}

\begin{conjecture}[{\cite[4.31]{H23}}] \label{conj:hu2}
If $[S/I] \in \Hilb^d(\AA^3)$ is smoothable with $\dim_{\kk} T(I) = 3d+6$, then the singularity at $[S/I] \in \Hilb^d(\AA^3)$ is smoothly equivalent to the vertex of a cone over the Grassmannian $\Gr(2,6) \hookrightarrow \mathbb{P}^{14}$ in its Pl\"ucker embedding.
\end{conjecture}

\Cref{conj:hu1} is a strengthening of the gap predicted by the parity conjecture \cite{MNOP}. We prove this conjecture for all monomial ideals, without any restriction on $d$, see \Cref{ThmSingularMonomial}. 
One class of ideals $I$ for which  $\dim_{\kk} T(I) = 3d+6$ are the \textbf{tripod ideals}
(\Cref{example_tripod_1}).
We establish \Cref{conj:hu2} for all tripod ideals (\Cref{ref:surplusSixIsGrCone:thm}).

Hu also conjectured a classification of Borel-fixed ideals with tangent
    space dimension $3d + 6$ in characteristic $0$ (see \cite[Conjecture
    4.27-4.29]{H23}). 
We prove these conjectures, 
and we also verify \Cref{conj:hu2} for Borel-fixed ideals 
(\Cref{theorem_strongly_stable_3d+6}, \Cref{ref:surplusSixIsGrCone:thm}).

Our results on \Cref{conj:hu2} are  obtained using linkage;
however, we also show that
 linkage cannot yield a full proof of the conjecture    (\Cref{ex:monomialSurplusSixNOTlinkedToKatz}).
As a byproduct of our analysis, we obtain a new criterion  for determining if an ideal is not in the linkage class of any monomial ideal, 
thus providing  a partial answer to the question posed in~\cite{Huneke_Ulrich__Monomials}.
 In particular, this yields 
 a new criterion  for determining if an ideal is not in the linkage class of a complete intersection;
very few such criteria are known \cite{Huneke_Ulrich__Structure_of_linkage}.

\begin{prop}[\Cref{corollary_linkage_obstruction}]
    If $[S/I]\in \Hilb^d(\mathbb{A}^3)$ is a point for which $\dim_{\kk} T(I) \not\equiv d \bmod 2$, then $I$ is not in the  linkage class of
 any homogeneous ideal (and therefore not linked to any monomial
    ideal). In particular,  $I$  is not licci.
\end{prop}
    By  work of Giovenzana-Giovenzana-Graffeo-Lella~\cite{Graffeo_Giovenzana_Lella,
    Graffeo_Giovenzana_Lella_Two}, concrete examples of ideals satisfying the hypothesis of the above theorem are known, see  \Cref{example_GGGL}.

\subsection{Further directions}
Our work opens up several avenues for further exploration, and we  highlight a few of them.

\subsubsection{The shape of the staircase:} Singularizing triples capture the idea expressed in~\cite{Rezaee__Conjectural_most_singular, H23} that the ``shape'' of the staircase $E_I$ reflects the geometry of the monomial point $[S/I]\in \Hilb^d(\mathbb{A}^3)$. The combinatorics of singularizing triples is intriguing in its own right: for instance, we believe that the singular monomial ideals with tangent space dimension $3d+6$ should also be classifiable and that all their singularizing triples share a common pair of socle elements.

\subsubsection{Generating functions} \label{subsecGenFunc}
Let $P_3(d)$ denote the number of monomial points of $\Hilb^d(\AA^3)$,
equivalently, the number of plane  partitions of $d$. MacMahon
gave the famous formula
\[
    \sum_{d=0}^\infty P_3(d)q^d = \prod_{i\geq 1} \frac{1}{(1-q^i)^{i}}.
\]
The singularities of monomial points, for example the tangent
space dimension, yield a refinement of this formula. A plausible first step to obtain it, would be to
take $P^{\sm}_3(d)$ to be the number of \emph{smooth} monomial ideals in
$\Hilb^d(\AA^3)$ and try to determine a closed formula for the generating
series $\sum_{d\geq 0}P^{\sm}_3(d)q^d$.
By \Cref{RemarkProduceSmooth},
we can generate all smooth monomial ideals up to a given $d$  directly (that is, without computing all monomials ideals). 
The first fourteen  values of $P^{\sm}_3(d)$ are as follows
$$
1, 3, 6, 12, 21, 36, 58, 91, 138, 204, 300, 417, 597, 816. 
$$
As far as we know, this sequence does not appear in the combinatorial literature.

\subsubsection{Structure of licci ideals}
Recent work by Weyman, in collaboration with Guerrieri, Ni, and others~\cite{Weyman__Three_takes, Weyman__Mapping, GNW__Higher_structure, NW__bigradings, GNW__ADE}, provides a theory of higher structure maps for licci ideals in codimension three.
It would be very interesting to understand the relationship between this theory and our work.

\section*{Acknowledgements}
The authors would like to thank Craig Huneke, Xianglong Ni, Claudia Polini, Bernd Ulrich, and Jerzy Weyman for several helpful conversations. 
We are very grateful to the anonymous referee for carefully reading our paper and for providing many valuable suggestions.
We  would like to thank University of Warsaw for hospitality in April 2024, where this project was finally shaped.\footnote{We would also like to give special thanks to Cofix for keeping us caffeinated.}

\section*{Funding}
Joachim Jelisiejew was supported by National Science Centre grant 2023/50/E/ST1/00336. 
Ritvik Ramkumar was partially supported by NSF grant DMS-2401462. 
Alessio Sammartano  was  supported by the grants PRIN 2020355B8Y,
PRIN 2022K48YYP,
and INdAM – GNSAGA  CUP E55F22000270001.

\section{Preliminaries} \label{sec:prelim} 

    Throughout the paper, we work over a field $\kk$.

    \subsection{Zero-dimensional algebras}\label{ssec:zerodimensional}
Throughout this subsection, let $R$ be a finite $\kk$-algebra. In particular, $R$ is automatically Cohen-Macaulay.
Let $M$ be a finite $R$-module.
Its \textbf{degree} is  $\deg M := \dim_{\kk} M$, and
its \textbf{dual module} is $M^{\vee} :=
\Hom_{\kk}(M, \kk)$, with the $R$-module structure given by $(r\cdot f)(m) :=
 f(rm)$ for $r\in R$, $m\in M$, and $f\in M^{\vee}$.

The canonical, or dualizing, module for $R$ is $\omega_R := R^{\vee}$.
By~\cite[Proposition~21.1]{EisView}, there is a unique (up to unique isomorphism) dualizing functor on the category of finitely generated $R$-modules. 
This implies that    the functors $\Hom_R(-, \omega_R)$ and $(-)^{\vee}$ are isomorphic. 
If $R$ is Gorenstein, then $\omega_R  \simeq R$, and thus $\Hom_R(-, R)$ and  $(-)^{\vee}$ are isomorphic.

Suppose that $R$ is a quotient of a polynomial ring $S$,    
and let $M$ be a finite $R$-module. 
By~\cite[Corollary~3.3.9]{BrunsHerzog}, we have
 $M^{\vee}  \simeq \Ext^{\dim S}_S(M, S)$. 
 Moreover, a free    resolution $F_{\bullet}\to M$ of length $\dim S$ yields a free resolution
    $F_{\bullet}^{*}$ of $M^{\vee}$, where $F_i^{*}$ denotes the free $S$-module
    $\Hom_S(F_i, S)$. This observation yields the following.

    \begin{lemma}\label{ref:transposition:lem}
        Let $M$, $N$ be finite degree $S$-modules. The $S$-modules
        $\Ext^i_S(M, N)$ and $\Ext^i_S(N^{\vee}, M^{\vee})$ are isomorphic.
        Morever, $\sum_{i=0}^{\dim S} (-1)^i \dim_\kk \Ext_S^i(M,N)=0$.
    \end{lemma}
    \begin{proof}
Let $F_{\bullet}$, $G_{\bullet}$ be  free resolutions of length $\dim S$ of $M$, $N$, respectively. 
Elements of $\Ext^i_S(M, N)$ are chain complex maps $f\colon F_{\bullet}[-i]\to G_{\bullet}$ up to null homotopy.
Such a map can be transposed to yield $f^\intercal[-i]\colon G^{\vee}_{\bullet}[-i] \to F^{\vee}_{\bullet}$. The transpose is an involution and
        null-homotopic maps are sent (surjectively) to null-homotopic ones, so
        $f^\intercal[-i]$ yields a well defined element of $\Ext^i_S(N^{\vee},
        M^{\vee})$. The map $f\mapsto f^\intercal[-i]$ is the required isomorphism.
        
Since $M$ has finite degree, its resolution satisfies $\sum (-1)^i \mathrm{rk} (F_i) = 0$.
Since $N$ has finite degree, for any finite free $S$-module $F$ we have $\dim_\kk \Hom(F,N) = \mathrm{rk}(F) \dim_\kk N$.
We conclude that
\begin{align*}
\sum_{i=0}^{\dim S}(-1)^i \dim_\kk \Ext_S^i(M,N) &
=  \sum_{i=0}^{\dim S}(-1)^i \dim_\kk H^i( \Hom (F_\bullet, N))
=  \sum_{i=0}^{\dim S}(-1)^i \dim_\kk  \Hom (F_i, N)\\
& = 
\sum_{i=0}^{\dim S} (-1)^i \mathrm{rk}(F_i) (\dim_\kk N)
				= \dim_\kk N \left(\sum_{i=0}^{\dim S}(-1)^i \mathrm{rk}(F_i)\right) = 0. \qedhere
\end{align*}
        \end{proof}
    
We will also need the following observations about the Tor functor. 
Recall that, for cyclic $S$-modules $S/I$, $S/J$, we have $\Tor_1^S(S/I,
 S/J) \simeq \frac{I\cap J}{IJ}$. 
 In the present paper, the Tor functor appears in connection to the Ext functor via the following lemma.

\begin{lemma}\label{ref:GorensteinTor:lem}
Let $R$ be a finite quotient of a polynomial ring $S$ and let $M$ be a finitely generated $S$-module. 
Then, $\Ext_S^i(M, \omega_R) \simeq (\Tor_i^S(M, R))^{\vee}$, where $(-)^{\vee}$ is applicable since $\Tor_i^S(M, R)$ is an $R$-module.
    \end{lemma}
    \begin{proof}
Let $F_{\bullet}$ be a free resolution of $M$ and consider the complex $F_{\bullet}\otimes_S R$. 
The groups $\Ext_S^i(M, \omega_R)$ arise by applying $\Hom_R(-, \omega_R)$ to this  complex and taking homology,
  while $(\Tor_i^S(M, R))^{\vee}$ arise by first taking homology and then applying $\Hom_R(-, \omega_R) \simeq (-)^{\vee}$. Since the functor $\Hom_R(-, \omega_R)$ is exact~\cite[Proposition~21.2]{EisView}, it commutes with taking homology, giving the desired result.
\end{proof}

\subsection{Tangent spaces and abstract deformation functors}  \label{ssec:defFunct}
We will briefly review the theory of abstract deformation functors. For more details, we refer to~\cite[\S7]{hartshorne_deformation_theory}, which operates under additional assumptions, and to the general theory in~\cite{fantechi_et_al_fundamental_ag, Sernesi__Deformations} for a broader framework. 


        \newcommand{\cM}{\mathcal{M}}
        \newcommand{\cN}{\mathcal{N}}
Let $S$ be a fixed Noetherian $\kk$-algebra and $M$ be a finitely generated $S$-module.
Let $\mathbf{Art}$ denote the category of local finite $\kk$-algebras $(A,\mm)$ with residue field $\kk$.
The functor $\Def_M\colon \mathbf{Art}\to \mathbf{Set}$ associates to a local finite $\kk$-algebra $(A,\mm)$ the set
 $\Def_M(A) = \left\{ (\cM, \iota) \right\}/\mathrm{iso}$, 
where $\cM$ is a finitely generated $(S\otimes_{\kk} A)$-module, flat over $A$,
 and $\iota$ is an isomorphism $\iota\colon \cM/\mm \cM \simeq M$.
        The functor $\Def_M$ admits a tangent space
        isomorphic to
        $\Ext^1_S(M, M)$ and a (complete) obstruction theory with obstruction
        group $\Ext^2_S(M, M)$, see \cite[15]{hartshorne_deformation_theory}, \cite[VI.1.3]{Eisenbud_Harris__Geometry_of_schemes}  and 
\cite[Example 6.3.7(2)]{fantechi_et_al_fundamental_ag}.
In the special case when $S = \kk[x_1, \ldots, x_n]$ and $S/I$ a finite quotient algebra,  the tangent space to the abstract deformation functor $\Def_{S/I}$ coincides with the tangent space to $[S/I] \in \Hilb(\mathbb{A}^n)$ and is given by
\[
T(I) := T\Def_{S/I} = \Ext^1_S(S/I, S/I) \simeq \Hom_S(I, S/I).
\]

There is  a flag (or nested) analogue of the abstract deformation functor (see~\cite[\S4.5]{Sernesi__Deformations}), which we will use only in \Cref{ref:smoothnessOfSmoothableBroken:cor}.
Given a 
        surjection $M\onto N$, the functor $\Def_{M \onto N} \colon \mathbf{Art}\to \mathbf{Set}$ associates to a local finite $\kk$-algebra $(A,\mm)$ the set $\{(\cM \onto \mathcal{N}, \iota_\cM, \iota_\cN)\}/\text{iso}$ where 
        $\iota_M  \colon \cM/\mm \cM \simeq M$ and $\iota_\cN \colon\cN/\mm \cN \simeq N$ are isomorphisms
        such that
        
\begin{equation}\label{DefMN_diagram}
\begin{tikzcd}
\mathcal{M} \arrow[two heads]{r} \arrow{d}& 
\mathcal{N} \arrow{d} \\
\cM/\mm \cM \overset{\iota_\cM}{\longrightarrow} M \arrow[two heads]{r} & N \overset{\iota_\cN}{\longrightarrow} \cN/\mm \cN
\end{tikzcd}
\end{equation}
commutes.
The tangent space to $\Def_{M\onto N}$ is
            \[
                \begin{tikzpicture}[commutative diagrams/every diagram]
                    \matrix[matrix of math nodes, name=m, commutative
                    diagrams/every cell]{
                        T\Def_{M\onto N} &     \Ext^1_S(M, M)\\
                        \Ext^1_S(N, N) & \Ext^1_S(M, N)\\
                    };
                    \path[commutative diagrams/.cd, every arrow, every label]
                    (m-1-1) edge node {$\pi_{M}$} (m-1-2)
                    (m-1-1) edge [left] node {$\pi_{N}$} (m-2-1)
                    (m-2-1) edge (m-2-2)
                    (m-1-2) edge (m-2-2);
                    \begin{scope}[shift=($(m-1-1)!.15!(m-2-2)$)]
                        \draw[line width=0.3mm] +(-0.3, -0.3) -- + (0, -0.3) --
                        +(0,0);
                    \end{scope}
                \end{tikzpicture}
            \]
            where $\pi_M$ and $\pi_N$ are the tangent maps for
            $\Def_{M\onto N}\to \Def_M$ and  $\Def_{M\onto N}\to \Def_N$, respectively,
            see \cite[Proposition 2.1]{MotiveNestedQuot}.
            Assuming that $\Ext^1_S(M, N)$ is the pushout in the above
            diagram,
             the flag
            deformation functor admits an obstruction theory, see~\cite[Appendix]{jelisiejew_Mandziuk}.

            In analogy with $\Def_{M\onto N}$, given an inclusion $K \subseteq M$ we can also define the functor $\Def_{K\subseteq M}$. More precisely, the functor $\Def_{K\subseteq M} \colon \mathbf{Art}\to \mathbf{Set}$ associates to an Artinian $\kk$-algebra $A$ the set of (isomorphism
            classes of) triples $(\cK \subseteq \cM, \iota_\cK, \iota_\cM)$ that satisfy a diagram analogous to \eqref{DefMN_diagram}. 
By the local criterion for flatness, the surjection $\cM\to \cM/\cK$ yields an element of
$\Def_{M\onto M/K}$ and we obtain $\Def_{K\subseteq M} \simeq \Def_{M\onto M/K}$.

\subsection{Linkage}\label{ssec:linkage}

        Linkage is a useful equivalence relation on Cohen-Macaulay
        \emph{ideals}, see~\cite{Peskine_Szpiro__Liaison, huneke_ulrich_cis,
        Huneke_Ulrich__Structure_of_linkage}. 
We will be interested in unmixed ideals of codimension three  in a regular ambient ring $S$. 
The ring $S$ is either local or standard graded, in the latter case all ideals considered are graded. 
Except for \Cref{sec:licci}, the considered ideals cut out zero-dimensional
schemes.
We denote by $\ua$, $\ub$  regular sequences  in $S$. 

 \begin{lemma}[
         {\cite{Peskine_Szpiro__Liaison}, \cite[Proposition 2.5]{Huneke_Ulrich__Structure_of_linkage}}]
 \label{ref:mainlinkage:lem}
Let $I\subseteq S$ be an unmixed ideal of codimension three, and $\ua \subseteq I$ a regular sequence of length three.
The ideal $(\ua : I) = \left\{ s\in S\ |\ sI \subseteq (\ua) \right\}$ is also unmixed of codimension three, and we have
 $I = (\ua : (\ua : I))$ and $\omega_{S/I} \simeq (\ua : I)/(\ua)$.
\end{lemma}

With notation as in Lemma \ref{ref:mainlinkage:lem}, 
the ideal $(\ua : I)$
is called the \textbf{link} of $I$ with respect to $\ua$.
Two
ideals $I$ and $J$ are \textbf{linked} if $J$ is the link of $I$ with respect
to some regular sequence. We say that $I$ and $J$ are in the same
\textbf{linkage class} if there is a chain of links from $I$ to $J$.
An ideal $I \subseteq S$ is said to be \textbf{licci} if it is in the linkage class of a complete intersection.

For a finitely generated Cohen-Macaulay $S$-module $M$ of codimension three, the \textbf{dual} $M^{\vee}$ is defined as $\Ext^3_S(M, S)$. 
 By~\cite[Proposition 3.3.3(b)(ii) and Corollary~3.3.9]{BrunsHerzog}, the operation $(-)^{\vee}$ is involutive, preserves being Cohen-Macaulay of codimension three, and, if $M$ is zero-dimensional, agrees with the definition in \Cref{ssec:zerodimensional} above.
If $I$ is an ideal such that $S/I$ is Cohen-Macaulay of codimension three, then the canonical module is $\omega_{S/I} = (S/I)^{\vee}$.

    \subsection{Monomial ideals}
We fix some of the notation and review the interpretation of tangent vectors to monomial ideals as bounded connected components, as developed in \cite{RS22}. 
The general linear group $\GL(n)$ acts on $S=\kk[x_1,\dots,x_n]$ by a change of coordinates, which induces an action on $\Hilb^{d}(\AA^n)$. 
We fix the maximal torus to be the subgroup of diagonal matrices and the Borel subroup to be the set of upper triangular matrices in $\GL(n)$. 
It is well known that an ideal $I$ is fixed by the maximal torus if and only if it is a monomial ideal.

\begin{definitions}
A \textbf{path} between  $\bfa, \bfb \in \Z^n$ is a sequence
$\bfa=\bfc^{0}, \bfc^{1},\ldots,\bfc^{m-1}, \bfc^{m}=\bfb$ of points of $\Z^n$ 
such that  $ \Vert \bfc^{i+1}-\bfc^{i} \Vert =1$ for all $i$,
where  $\Vert \bfd\Vert=\sum_{j=1}^n|\bfd_j|$.

A subset $U \subseteq \Z^n$ is said to be \textbf{connected} if it is non-empty and for any two points  $\bfa,\bfb \in U$ there is a path between them contained in $U$.
Given a subset $V \subseteq \Z^n$, a maximal connected subset $U\subseteq V$ is called a \textbf{connected component}.
A  subset $U\subseteq \Z^n$ is  \textbf{bounded} if it is finite.
\end{definitions}

Let $I \subseteq \N^n$ be (the set of exponent vectors of) a monomial ideal,
 and $\bfa \in \Z^n$.
A  connected component $U$ of $ (I+\bfa) \setminus \, I$ is  bounded if and only if $ U \subseteq \N^n$.

\begin{prop}[{\cite[Proposition 1.5]{RS22}}] \label{proposition_basis} Let $I$
    be a  cofinite monomial ideal  and $\bfa \in \Z^n$. 
    The number of bounded connected
    components of the set $(I+\bfa) \setminus \, I$ is equal to
    $\dim_{\kk}\Hom_S(I,S/I)_{\bfa}$, where $(-)_{\bfa}$ denotes the degree $\bfa$-component.
\end{prop}

\subsection{Macaulay's inverse systems} \label{ssec:macInv}
Macaulay's inverse system, also known as apolarity, is a standard way to
construct zero-dimensional schemes. It is especially effective for Gorenstein
ones. Some references are~\cite{Macaulay_inverse_systems,
iarrobino_compressed_artin, EliasRossiShortGorenstein,
Jel_classifying}. We will use it only for examples, so we give only a brief
overview.

Let $S = \kk[x, y, z]$ and $P = \kk[X, Y, Z]$. 
We view $P$ as an
$S$-module via the \emph{contraction} action:
\[
    x\circ X^{\bfa_1}Y^{\bfa_2}Z^{\bfa_3} = \begin{cases}
        X^{\bfa_1 - 1}Y^{\bfa_2}Z^{\bfa_3} & \mbox{if } \bfa_1 > 0,\\
        0 & \mbox{otherwise},
    \end{cases}
\]
and similarly for $y$ and $z$ actions. 
If one views $P$ as a divided power
algebra, then $S$ acts by derivations~\cite[Appendix~A]{iakanev},
\cite[\S2.1]{Jel_classifying}. 
For every $f_1, \ldots, f_r\in P = \kk[X,Y,Z]$,
 we can consider the annihilator $\Ann(f_1, \ldots ,f_r) \subseteq S$ 
 of the submodule $S f_1 + \cdots Sf_r \subseteq P$.
For example, $\Ann(X^2 + YZ) = \left( x^2 - yz, xz, xy,
y^2, z^2\right)$.

\begin{thm}[Macaulay's theorem~\cite{Macaulay_inverse_systems}, formulated in codimension
    three]\label{ref:Macaulay:thm}
If $f_1, \ldots , f_r\in \kk[X,Y,Z],$ the quotient  $S/\Ann(f_1, \ldots ,f_r)$ is a finite local algebra. 
If $r = 1$, then it is Gorenstein.

Conversely, for every finite local  algebra $S/I$ there exist  $f_1, \ldots , f_r\in P$ such that $I = \Ann(f_1, \ldots, f_r)$. 
One can take $r = \degold(\soc(S/I))$. 
In particular,  if $S/I$ is Gorenstein, then there exists  $f\in P$ such that $I = \Ann(f)$.
\end{thm}
 \Cref{ref:Macaulay:thm} is particularly useful when 
the dimension of $\soc(S/I)$, that is, the \textbf{type} of $S/I$, is much smaller than the
number of generators of $I$.

\section{Broken Gorenstein structures}\label{sec:BGS}

In this section, we develop the theory of broken Gorenstein algebras.
 The main goal is to prove that if a smoothable algebra $R = \kk[x, y, z]/I$ admits a broken Gorenstein  structure, then the corresponding point $[R]\in \Hilb^d(\AA^3)$ is smooth.
We also introduce the bicanonical module and prove a structure theorem for broken Gorenstein algebras without flips.

We start by giving explicit descriptions of 1- and 2-broken Gorenstein algebras. In particular, a $1$-broken Gorenstein algebra, regardless of flips, is simply an extension
   \[
       0\to R_1 \to R\to R_0\to 0,
   \]
   with $R_0$, $R_1$ cyclic $R$-modules, corresponding to Gorenstein algebras. A
   $2$-broken Gorenstein algebra structure is a diagram
   \[
       \begin{tikzcd}
           R_1\ar[d, hook]\\
           \cK\ar[r, hook]\ar[d, two heads] & R \ar[r, two heads] & R_0\\
           R_2
       \end{tikzcd}
   \]
 where
$R_0$, $R_1$, $R_2$ are cyclic $R$-modules corresponding to Gorenstein algebras
and $\cK$
 is either cyclic (no flip) or cocyclic (flip).

    \subsection{The bicanonical module}\label{ssec:bicanonical}
As noted in the introduction, the bicanonical module plays a crucial role in the proof of \Cref{ref:canitbetrue:thm}, so we will introduce it now. In this section, we do not impose any ``codimension three" assumptions; instead, we define bicanonical modules in a broader context, as they are of general interest.

\begin{definition}
 Let $R$ be a finite $\kk$-algebra and $M$ an $R$-module. The \textbf{symmetric square} of $M$ is $\Sym_R^2 M$ and the \textbf{bicanonical module} for $R$ is defined to be $\Sym_R^2 \omega_R$.
\end{definition}

Recall that the symmetric square $\Sym_R^2 \omega_R$ is obtained from the ``usual" symmetric square $\Sym_{\kk}^2 \omega_R$ by imposing relations of the form $(r\varphi_1) \cdot \varphi_2 = \varphi_1 \cdot (r\varphi_2)$ for all $\varphi_1, \varphi_2 \in \omega_R$ and $r \in R$. The degree of $\Sym_{\kk}^2 \omega_R$ is always equal to $\binom{\degold(R)+1}{2}$.
By contrast, computing the degree of $\Sym_R^2 \omega_R$ is much more complex. We will see below that, under favorable conditions, this degree can equal $\degold(R)$.

The following result provides a way to bound the degree of the bicanonical module for an algebra that has a broken Gorenstein  structure.

    \begin{prop}\label{ref:induction:prop}
        Let $R$ be a finite $\kk$-algebra equipped with a short exact
        sequence of $R$-modules
        \[
            0\to \cK \to R\to R_0\to 0,
        \]
        where $R_0$ is a Gorenstein algebra.
        If the $R$-module $\cK$ is cyclic or cocyclic, then
        \[
            \degold (\Sym^2_{R} \omega_R) \leq \degold (\Sym^2_{R} (\cK^{\vee})) + \degold R_0.
        \]
    \end{prop}
    \begin{proof}
        By assumption, we have an exact sequence of $R$-modules $0\to \cK \to
        R\to R_0\to 0$,
        which dualizes to an exact sequence
        \[
            0\to \omega_{R_0}\to \omega_{R}\to \cK^{\vee}\to 0
        \]
        of $R$-modules. 
        Applying $\Sym^2_R(-)$, we obtain an exact
        sequence \cite[Proposition 4, p. A III.69]{Bourbaki_Algebra_1_to_3}
        \[ 0 \to
            \omega_{R_0}\cdot \omega_{R}\to \Sym^2_{R} \omega_R \to
            \Sym^2_{R}(\cK^{\vee}) \to 0.
        \]
        Hence, $\degold (\Sym^2_R \omega_R) = \degold (\Sym^2_{R}(\cK^{\vee})) + \degold
        (\omega_{R_0} \cdot \omega_R)$. It remains to
        bound the second summand.

        Since $R_0$ is a zero-dimensional Gorenstein algebra, the module $\omega_{R_0}$ is cyclic and generated by some $g \in \omega_{R_0}$. Consequently, there is a surjective map
        $p\colon\omega_R\to \omega_{R_0}\cdot
        \omega_R\subseteq \Sym^2_R \omega_R$ which sends $\varphi$ to $g\cdot
        \varphi$. By definition, the module $\omega_{R_0} = \left( R/\cK
        \right)^{\vee}$ is annihilated by the ideal $\cK\subseteq R$. Thus, $\omega_{R_0}\cdot
        \omega_R$ is also annihilated by $\cK$ and the map $p$ factors to a
        surjective map
        \[
            \frac{\omega_R}{\cK  \omega_R}\onto \omega_{R_0}\cdot \omega_R.
        \]
        We will prove that
        \begin{equation}\label{eq:mainEquality}
            \degold (\cK \omega_R) = \degold \cK,
        \end{equation}
        which will show
        that $\degold (\omega_R/\cK\omega_R) = \degold R - \degold \cK = \degold R_0$ and
        thereby conclude the proof.

        Consider the sequence $0\to \cK \omega_R \to \omega_R \to \omega_R/(\cK
        \omega_R)\to 0$ and dualize it to obtain
        \[
            0\to J \to R\to \left( \cK\omega_R \right)^{\vee}\to 0,
        \]
        where  $J$ is the ideal   $(\omega_R/(\cK \omega_R))^\vee$. In particular, 
        \begin{align*}
            J 
             = \left\{ r\in R:  \varphi(r)=0 \text{ for all } \varphi\in \cK\omega_R \right\} 
            &= \left\{ r\in R:  \varphi(rr') = 0  \text{ for all } \varphi\in \omega_R \text{ and } r'\in \cK  \right\}  \\
            & = \left\{ r\in
           R:  rr' = 0 \text{ for all } r'\in \cK  \right\}.
        \end{align*}
        In particular, $J = \Ann_R(\cK)$ and it follows that
        $\degold((\cK\omega_R)^{\vee}) = \degold(R/\Ann_R(\cK))$.
        If $\cK$ is cyclic or
        cocyclic,
        then $R/\Ann_R(\cK)$ is isomorphic to $\cK$ or $\cK^{\vee}$,
        respectively. In both cases, we
        obtain \Cref{eq:mainEquality}, as desired.
    \end{proof}

    \begin{cor}\label{ref:brokenGorensteinCanonical:cor}
        Let $R$ be a finite $\kk$-algebra with a broken Gorenstein
        structure. Then $\degold \Sym^2_R \omega_R$ is at most $\degold (R)$.
        Moreover, $\degold(R)$ is  equal to $\degold \Sym^2_R
        R$.
    \end{cor}
    \begin{proof}
If $R$ is Gorenstein, we have $\Sym_R^2 \omega_R \simeq \Sym_R^2 R \simeq R$. Thus, the claim holds for $0$-broken Gorenstein algebras. 
For $k \geq 1$, we proceed by induction using \Cref{ref:induction:prop} and the fact that $\cK^{\vee} \cong R/\Ann(\cK)$ or $\cK^{\vee} \cong \omega_{R/\Ann(\cK)}$. 
The assertion that $\Sym_R^2 R \simeq R$ is immediate and is included here for  reference.
    \end{proof}

\begin{definition}
A homomorphism $\varphi\colon \omega_R \to R$ is said to be  \textbf{symmetric}
    if $\varphi^{\intercal}\colon \omega_R = R^{\vee}\to \omega_R^{\vee}= R$ is equal to $\varphi$. 
    We denote by $\Hom_{\kk}^{\sym}(\omega_R, R)$ (respectively, by
     $\Hom_R^{\sym}(\omega_R,R)$),
the $\kk$-subspace of  $\Hom_{\kk}(\omega_R, R)$
(respectively, the $R$-submodule of 
     $\Hom_R(\omega_R,R)$) consisting of symmetric homomorphisms.
    \end{definition}

The bicanonical module of a finite $\kk$-algebra $R$  admits an interpretation in terms of maps. 
Assuming $\mathrm{char}(\Bbbk) \ne 2$,
since $\Sym^2_R \omega_R$ is an image of $\Sym^2_{\kk} \omega_R$, 
its dual $(\Sym^2_R \omega_R)^{\vee} =
    \Hom_{R}(\Sym^2_R \omega_R, \omega_R)$ is a subspace of $\Sym^2_{\kk}
    \omega_R^{\vee} = \Sym^2_{\kk} R \simeq \Hom_{\kk}^{\sym}(\omega_R, R)$.

        \begin{lemma}\label{ref:compatibility:lem}
            Let $R$ be a finite $\kk$-algebra, and assume that $\kk$ has characteristic different from 2.
            Under the identification $\Sym^2_{\kk} \omega_R^{\vee}  \simeq
            \Hom_{\kk}^{\sym}(\omega_R, R)$, the module $(\Sym^2_R \omega_R)^{\vee}$
            is isomorphic to $\Hom_R^{\sym}(\omega_R, R)$.
        \end{lemma}
        \begin{proof}
            Let $r_1\odot r_2$ denote the class of $r_1\otimes r_2$ in
            $\Sym^2_{\kk} R$.
            The subspace $(\Sym^2_R \omega_R)^{\vee}$ consists of elements $\sum_i
            r_{1i}\odot r_{2i}$ such that, for all $r\in R$ and $f,g\in
            \omega_R$, the following condition holds:
            \[
                \left\langle (r f)\odot g - f\odot (rg), \sum_i
                r_{1i}\odot r_{2i}\right\rangle = 0.
            \]
            This means that, for every $f,g$ and $r$, we have
            \[
                \sum_i (rf)(r_{1i})\cdot g(r_{2i}) + (rf)(r_{2i})\cdot
                g(r_{1i}) = \sum_{i} f(r_{1i})\cdot g(rr_{2i}) +
                f(r_{2i}) \cdot g(rr_{1i}).
            \]
            This holds for every functional $g$, which implies that
            \begin{equation}\label{eq:functionals}
                \sum_i (rf)(r_{1i})\cdot r_{2i} + (rf)(r_{2i})\cdot
                r_{1i} = \sum_{i} f(r_{1i})\cdot rr_{2i} +
                f(r_{2i}) \cdot rr_{1i}
            \end{equation}
            Let $\varphi\in \Hom_{\kk}^{\sym}(\omega_R, R)$ be the element
            corresponding to $\sum_i r_{1i}\odot r_{2i}$ above. The value
            $\varphi(rf)$ is the left hand
            side of~\eqref{eq:functionals}, while $r\varphi(f)$ is the right hand
            side. Equality~\eqref{eq:functionals} shows that $\varphi$ is
            $R$-linear. The argument can be reserved.
        \end{proof}
          
        We now give a sample computation of $\Hom^{\sym}_R(\omega_R, R)$ for a monomial algebra.

        \begin{example}

           Let $R = \kk[x, y]/(x^2, xy^2, y^5) = \kk[x, y]/\Ann(Y^4, XY)$ with $\degold(R) = 7$. 
            In this case, $\omega_R$ is generated by $Y^4$ and $XY$, which correspond to the
            functionals dual to $y^4$ and $xy$ in the monomial basis, respectively.

            The vector space $\Hom_R(\omega_R, R)$ has dimension $9$ and is spanned
            by the $9$ homomorphisms $\varphi$ given by:
            \[
                \begin{array}{c | c c c c c c c c c}
                    \varphi(XY) & x&x\,y&y^{3}&y^{4}&0&0&0&0&0\\
                    \varphi(Y^4) & 0&0&0&0&x&x\,y&y^{2}&y^{3}&y^{4}
                \end{array}
            \]
            The subspace of symmetric homomorphisms  $\Hom^{\sym}_R(\omega_R, R)$ is $7$-dimensional (as will be shown in \Cref{ref:HaimanLikeBundle:prop}) and is spanned by
            \[
                \begin{array}{c | c c c c c c c c c}
                    \varphi(XY) & x&x\,y&y^{3}&y^{4}&0&0&0\\
                    \varphi(Y^4) & 0&0&x&x\,y&y^{2}&y^{3}&y^{4}
                \end{array}
            \]
        \end{example}
        The theory of bicanonical modules will be developed further in a
        subsequent paper.

    \subsection{Broken Gorenstein algebra structures: planar case}

    The definition of a broken Gorenstein  structure in
    \Cref{def:brokenGorenstein} may initially appear abstract and dry. To
    provide a more conceptual understanding, we start this section with two
    examples that are of independent interest.

    \begin{example}[Planar monomial
        ideals]\label{ex:monomialIdealsBrokenStructure}
        Let $R = \kk[x, y]/I$ be a finite $\kk$-algebra, with $I$ a monomial ideal.
        Write $I = (y^{e}, y^{e-1}x^{m_{e-1}}, \ldots , y^1x^{m_1},
        x^{m_0})$ with $m_{e-1} \leq \cdots \leq m_1 \leq m_0$.
        Then, $R$ has a broken Gorenstein  structure with no flips and
        with subquotients of the form $\kk[x]/(x^{m_i})$ with $i=0,1, \ldots
        ,e-1$. To see this, consider the chain of principal ideals
        \[
            0 = y^e R \subseteq y^{e-1} R\subseteq  \ldots \subseteq y
            R\subseteq R.
        \]
    \end{example}
The above broken Gorenstein structure is not unique. 
By replacing the roles of
$x$ and $y$ above, we get another one. Usually, there are (many) more than
two, because we can, for example, interchange the roles of $x$ and $y$ along the chain. 
One
concrete example is $I = (x, y)^3$ and the broken Gorenstein structure on $R =
\kk[x, y]/I$ given by
\[
    0 \subseteq Rxy \subseteq Rx \subseteq R,
\]
where the subquotients are $R/(x) \simeq \kk[y]/(y^3)$, $Rx/Rxy \simeq
\kk[x]/(x^2)$ and $Rxy \simeq \kk$.

More generally, all planar ideals admit a broken Gorenstein structure with no flips.

\begin{example}[Planar ideals] \label{example_planar_ideals}
We follow~\cite[Lemma~8.12]{hartshorne_deformation_theory}, which, in hindsight, points towards a
$1$-broken Gorenstein structure.
Let $R = \kk[x,y]/I$ be a finite $\kk$-algebra with the radical of $I$ equal to
$\mm = (x, y)$. 
Choose $g \in \mm^s - \mm^{s+1}$ with $s$ minimal. 
Then, up to a
change of coordinates, we may assume that the lowest degree form of $g$ is $g_0 = x^s + \cdots$.
Subtracting multiples of $g$ from itself, we may assume that $g$,
{considered as a polynomial in $x$}, is of degree $s$, with leading term $x^s$. In particular, we can write
$I = (g) + yI'$ where $I' = (I:y)$. This gives us an exact sequence
$$
0 \longrightarrow \kk[x,y]/I' \simeq yR \longrightarrow R \longrightarrow R/(y) \longrightarrow 0.
$$ 
Since $R/yR= \kk[x,y]/(x^s,y)$ is Gorenstein, we see that $R$ has a broken Gorenstein structure  if $\kk[x,y]/I'$ has one. Since $\degold (\kk[x,y]/I') < \degold (R)$, we may repeat this procedure iteratively to conclude that $R$ has a broken Gorenstein structure with no flips.
\end{example}

We take a small detour now, to observe that the bicanonical module yields a
    ``global'' invariant of the Hilbert scheme on the plane. 
    \begin{prop}\label{ref:HaimanLikeBundle:prop}
        There is a rank $d$ bundle on the Hilbert scheme
        $\Hilb^d(\AA^2)$, such that the fiber of this bundle over
        $[R]\in \Hilb^d(\AA^2)$ is isomorphic to the bicanonical module
        of $R$.
    \end{prop}
    \begin{proof}
        To simplify notation, let $H := \Hilb^d(\AA^2)$, and let
        $U$ be the universal bundle on $H$. The fiber of the dual bundle $U^{\vee}$ over a point 
        $[R]\in H$ is isomorphic to $\omega_R$. Let
        $B
        := \Sym^2_H U^{\vee}$ and note that its fiber over $[R] \in H$ is $\Sym^2_R \omega_R$, the bicanonical module of $R$. We need to show
        that $B$ is locally free of rank $d$. Since the Hilbert scheme is smooth and irreducible \cite{fogarty}, it suffices to show that for every closed point $[R]\in H$, we have $\degold \Sym^2_{R} \omega_R
        = d$ \cite[Corollary 11.19]{gortz_wedhorn_algebraic_geometry_I}. 
   The equality holds for Gorenstein $R$, and since the points
    corresponding to Gorenstein algebras form an open locus in $\Hilb^d(\AA^2)$, it holds
        generically. By the upper-semicontinuity of fiber dimension, we have $\degold \Sym^2_R \omega_R \geq
        d$ for every $R$.
        To prove equality, we again use upper-semicontinuity. It suffices to show that
        for  every algebra $R = \kk[x, y]/I$ with $I$ a monomial ideal, we have $\degold \Sym^2_R \omega_R \leq d$.
        This result follows from applying \Cref{ex:monomialIdealsBrokenStructure} and
        \Cref{ref:brokenGorensteinCanonical:cor}.
   \end{proof}

\begin{remark}
\Cref{ref:HaimanLikeBundle:prop} is particularly striking because the  bundle $B$ is torus-equivariant. 
While Haiman~\cite{Haiman_vanishing2} extensively studied the equivariant K-theory of $\Hilb^d(\AA^2)$, the bundle $B$ does not explicitly appear in the literature. 
It would  be interesting to relate $B$ with other notable bundles on the Hilbert scheme.
   \end{remark}

\subsection{Constructions of broken Gorenstein algebras and the necessity of flips}
The following example gives  an effective method for constructing
broken Gorenstein algebras.

 \begin{example}\label{exampleFlip}
Let $f\in \kk[X, Y, Z]$ and $g\in \kk[Y, Z]$ be polynomials. Let $S = \kk[x,
y, z]$, $R = S/\Ann(f, g)$ and $R_0 = S/\Ann(f)$. The kernel $\cK$ of $R\to
R_0$ is cocyclic (with cogenerator coming from $g$) and annihilated by $x$. Thus, by
\Cref{example_planar_ideals}, $R/\Ann(\cK)$ admits a broken Gorenstein structure
(without flips) and so $R$ admits a structure of a broken
    Gorenstein algebra (with flips).
\end{example}

\begin{remark}\label{RemPoonen}
Using   \Cref{exampleFlip} and 
 \Cref{ThmSmoothMonomialClassification}, 
we can verify \Cref{main_conjecture}  for algebras of degree $d \leq 6$.
There are finitely many  isomorphism types of such algebras, and they  are listed explicitly in  \cite{Poonen}.
A simple check shows that,
among the 
 algebras of embedding dimension at most 3 and degree at most 6,
 those  corresponding to smooth points satisfy the conditions of  \Cref{exampleFlip},
while those corresponding to singular points are defined by monomial ideals.
Thus, \Cref{exampleFlip} and 
 \Cref{ThmSmoothMonomialClassification} imply the equivalence
 $\eqref{it:main_conj:brokenGor}\Leftrightarrow\eqref{it:main_conj:smooth}$ in  \Cref{main_conjecture}.
 As already stated, the direction $\eqref{it:main_conj:licci}\Rightarrow\eqref{it:main_conj:smooth}$ is well known, 
 and the direction  $\eqref{it:main_conj:smooth}\Rightarrow\eqref{it:main_conj:licci}$ follows, for these algebras, assuming that $\kk$ has characteristic 0, from the fact that 
 ideals with small type and small deviation are licci
\cite[Theorem 6.2]{GNW__ADE}.
\end{remark}

The theory of broken Gorenstein algebras without flips is much easier, as seen
already in~\eqref{eq:noflipsFlag} and soon to be confirmed by \Cref{ref:NoFlipsLicci:thm}.
It is natural to wonder whether flips are necessary in
\Cref{def:brokenGorenstein}, that is, whether there exist smooth points with
broken Gorenstein structure that requires flips. The example below confirms this.

\begin{example}[An algebra with broken Gorenstein structure, but none without flips] \label{example_flips} 
Let $R =\kk[x,y,z]/I$ where $I = (yz,x^2z,xy^2-xz^2,x^2y,x^3+y^3,x^4,y^4,z^3)
= \Ann(X^3-Y^3, XY^2+XZ^2)$.
This is a graded algebra with Hilbert function $(1,3,5,2)$. 
We  first show that $R$ has a broken Gorenstein structure with flips.

Consider the submodule $\cK = (y^2-z^2,x^2)R$ and the natural exact sequence $0 \to \cK \to R \to R_0$ with $R_0 =\kk[x,y,z]/(yz,y^{2}-z^{2},x^{2},y^{3},z^{3})$.
The algebra $R_0$ is Gorenstein and has Hilbert function $(1,3,3,1)$. 
Since the $R$-module $\cK$ has Hilbert function $(0, 0, 2, 1)$, it is not cyclic. 
However, the dual module $\cK^{\vee}$ is cyclic, since it is isomorphic to $R' = \kk[x, y, z]/(z, x^2, xy, y^2)$. The algebra $R'$ admits a broken Gorenstein structure (without flips) by
 \Cref{ex:monomialIdealsBrokenStructure}. 
 Thus, $R$ admits a broken Gorenstein structure.

\smallskip
We now show that any broken Gorenstein structure on $R$ must have flips. 
If there were no flips, then we could find an exact sequence 
\[
0 \to \mathcal{K}  = fR \longrightarrow R \longrightarrow R_0 \longrightarrow 0
\]
such that $R_0 = R/(f) = \kk[x,y,z]/(I+f)$ is Gorenstein.
We claim  that, no matter how we choose $f\in \kk[x, y, z]$, we will always get a contradiction.
If $f\in (x, y, z)^2$, then the Hilbert function of $R/f$ has the form
$(1,3,\geq 4,*)$. Such a Hilbert function is not possible for a Gorenstein
algebra by~\cite[\S4]{EliasRossiShortGorenstein}
following~\cite[Proposition~1.9]{ia94}.
We conclude that $f\not\in(x, y, z)^2$, so it has a nontrivial linear part.
Write $f = l_1x + l_2y + l_3z + Q$ with $Q \in (x,y,z)^2$.
Observe that $x^3$, $xz^2$ form a basis of the degree three part of $R$.
We have the following equalities
\begin{itemize}
\item $x^2f = l_1 x^3 + l_2 x^2y + l_3 x^2 z +  x^2Q \equiv l_1 x^3 \bmod I$,
\item $z^2f = l_1 z^2x + l_2 yz^2 + l_3 z^3 +  z^2Q \equiv l_1 xz^2 \bmod I$,
\item $y^2f = l_1 xy^2 + l_2 y^3 + l_3 y^2z +  y^2Q \equiv  l_1xz^2 - l_2 x^3 \bmod I$,
\item $xyf = l_1 x^2y + l_2 xy^2 + l_3 xyz +  xyQ \equiv l_2 xz^2 \bmod I$.
\end{itemize}
In particular,  if $l_1 \ne 0$ or $l_2 \ne 0$ we get that $(x,y,z)^3 \subseteq
I+(f)  $, and that the Hilbert function of $R_0$ is either $(1,2,3)$ or
$(1,2,2)$, contradicting the fact that it should be Gorenstein. We conclude
that $f = z + Q$ with $Q\in (x, y, z)^2$.
Since $yz$ annihilates $R$, it follows that $((xf,yf,zf) + (x,y,z)^3)R$ is equal to $((xz,z^2) +(x,y,z)^3)R$. Thus, the quotient $R/(f)$ has Hilbert function
$(1,2,3,*)$, which is again impossible for a Gorenstein algebra by~\cite[\S4]{EliasRossiShortGorenstein}
following~\cite[Proposition~1.9]{ia94}.
\end{example}

   \subsection{Broken Gorenstein  implies smoothness}
The connection between bicanonical modules and broken Gorenstein structures is established by the following pivotal lemma. 
This lemma will ultimately allow us to provide upper bounds for the tangent space.

    Let $S$ be a fixed $\kk$-algebra. 
    Let $0\to \cK\to R\to R_0\to 0$ be a short exact sequence of $S$-modules,
    with $R$ and $R_0$ cyclic. The natural map $\Hom_S(R_0, R_0)\to
    \Hom_S(R, R_0)$ is an isomorphism. Consequently, the long exact sequence for $\Ext_S$
    yields the following exact sequence
    \begin{equation}\label{eq:lesExt}
        \begin{tikzcd}
            0\ar[r] & \Hom_{S}(\cK, R_0) \ar[r] & \Ext^1_{S}(R_0, R_0) \ar[r, "\varphi"]
            & \Ext^1_{S}(R, R_0)\ar[r] &
            \Ext^1_{S}(\cK, R_0).
        \end{tikzcd}
    \end{equation}

            \begin{lemma}[cokernel image]\label{ref:cokernel:lem}
                Let $R$ be a finite quotient of a polynomial ring $S$ and $0\to \cK\to R\to R_0\to 0$ be a
                short exact sequence with $R_0$ Gorenstein.
                In the setting of \Cref{eq:lesExt}, we have
                \[
                    \degold \coker \varphi - \degold \ker \varphi \leq \degold
                    \Sym^2_R \cK - \degold \cK.
                \]
                If $2$ is invertible in $\kk$ and $R_0$ has embedding dimension
                three, then equality holds.
            \end{lemma}
            \begin{proof}
                Let $R = S/J$ and $R_0 = S/I$. From the surjection $R\to R_0$,
                we get $I\supseteq J$.
                Since $R_0$ is Gorenstein, 
                the functors $(-)^{\vee}$ and $\Hom_{R_0}(-, R_0)$ are
                isomorphic on the category of $R_0$-modules,
                see~\Cref{ssec:zerodimensional}.
                Since $R_0$ is Gorenstein, this is an exact functor,
                 and  $\Ext^i_S(-, R_0)  \simeq \Tor_i^S(-, R_0)^{\vee}$ by  \Cref{ref:GorensteinTor:lem}. 
                 The
                map
                \[
                    \varphi\colon \Ext^1_S(R_0, R_0) \to \Ext^1_S(R, R_0)
                \]
                is thus the dual of $\varphi^\intercal\colon \Tor_1^S(R, R_0)\to \Tor_1^S(R_0, R_0)$, which
                in turn identifies with the natural map
                \[
                    \varphi^\intercal\colon \frac{J}{IJ} = \frac{I\cap J}{IJ}\to
                    \frac{I\cap I}{I^2} = \frac{I}{I^2}.
                \]
                The kernel and cokernel of this last map are $\frac{I^2\cap J}{IJ}$ and
                $\frac{I}{J+I^2}$, respectively. The intersection $I^2 \cap J$
                is not convenient to interpret directly, so we modify it slightly. 
                Recall that
                since $I \supseteq J$ we have $I^2 \supseteq I^2 \cap J \supseteq IJ$. Thus, we obtain
                \begin{align*} \label{eq:change}
\degold \left(\frac{I^2\cap J}{IJ}\right) = \degold \left(\frac{I^2}{IJ}\right) - \degold\left(\frac{I^2}{I^2\cap J}\right), \quad \text{ and } \quad 
                \degold \left(\frac{I}{J + I^2}\right) = \degold\left(\frac{I}{J}\right) - \degold\left(\frac{I^2 + J}{J}\right).
                \end{align*}
                The modules $I^2/(I^2 \cap J)$ and $(I^2 + J)/J$ are
                isomorphic, so we obtain
                \begin{equation}\label{eq:dualised}
                    \degold \coker\varphi - \degold \ker \varphi =
                    \degold \ker\varphi^\intercal - \degold \coker \varphi^\intercal =
                    \degold \frac{I^2}{IJ} - \degold
                    \frac{I}{J} = \degold \frac{I^2}{IJ} - \degold \cK.
                \end{equation}
                The ideal $I^2$ is the image of $\Sym ^2_S(I)$, and so $I^2/(IJ)$ is an image of
                $\Sym^2_R(I)/(J\cdot I)  \simeq \Sym^2_R(I/J) = \Sym^2_R \cK$,
                see~\cite[\href{https://stacks.math.columbia.edu/tag/00DO}{Tag
                00DO}]{stacks_project} for the isomorphism. Thus, the claim follows
                from \Cref{eq:dualised}.

                Suppose now that $1/2\in \kk$ and that $R_0$ has embedding
                dimension three. It follows from~\cite[Example,
                p.~209]{Simis_Vasconcelos__The_syzygies_of_conormal_module} that
                $I$ is syzygetic, so $\Sym^2_S(I)  \simeq I^2$.\
                Therefore, the map $\Sym^2_R(I/J)\to I^2/IJ$ is also an isomorphism and equality holds.
            \end{proof}

\begin{definition} \label{def_tangent_excess}
            Let $S = \kk[x_1, \ldots ,x_n]$.
            For an $S$-module $M$ of finite degree, we define the
            \textbf{smoothable tangent excess of $M$} (or \textbf{at $[M]$}) to be
            the number
            \[
                \delta_M := \degold Ext^1_S(M, M) - n\cdot\degold M.
            \]
\end{definition}

            \begin{thm}\label{ref:canitbetrue:thm}
                Let $S = \kk[x, y, z]$ and let  $R$ be a finite quotient algebra of $S$.
                Suppose that there is a short exact sequence
                \[
                    0\to \cK\to R\to R_0\to 0
                \]
                such that $\cK$ is either cyclic or cocyclic, and $R_0$ is
                Gorenstein.
                Then we have
                \begin{equation}\label{eq:tangent_excessChange}
                    \delta_R \leq \delta_{\cK} + 2\left( \degold (\Sym_R^2 \cK) -
                    \degold \cK \right).
                \end{equation}
                If equality holds in \Cref{eq:tangent_excessChange}, then the
                following also holds in the notation of \Cref{fig:les}:
                \begin{enumerate}
                    \item\label{it:canitbetrue1} The map $b' + \varphi\colon \Ext^1_S(R, R)\oplus
                        \Ext^1(R_0, R_0)\to \Ext^1_S(R, R_0)$ is surjective.
                    \item\label{it:canitbetrue2} The image of $d\colon \Ext^1_S(\cK, \cK)\to
                        \Ext^1_S(\cK, R)$ is contained in the image of
                        $c'\colon \Ext^1_S(R, R)\to \Ext^1_S(\cK, R)$.
                \end{enumerate}
            \end{thm}
            \begin{proof}
We will prove this by bounding the degree of $\Ext^1_S(R, R)$ from above. Consider \Cref{fig:les},  derived from three long exact sequences of $\Ext_S$ groups obtained from the exact sequence $0\to \cK\to R\to R_0\to 0$.
                \begin{figure}[h!]
                    \begin{tikzcd}
                        &0\ar[d]&0\ar[d]&0\ar[d]\\
                            0 \ar[r] & \Hom_{S}(R_0, \cK) \ar[r] \ar[d] & \Hom_{S}(R_0,
                                R) \ar[r] \ar[d] & \Hom_{S}(R_0, R_0) \ar[d,
                                "\simeq"] & \\
                            0 \ar[r] & \Hom_{S}(R, \cK) \ar[r] \ar[d] & \Hom_{S}(R,
                                R) \ar[r] \ar[d] & \Hom_{S}(R, R_0) \ar[d] & \\
                            0 \ar[r] & \Hom_{S}(\cK, \cK) \ar[r] \ar[d] & \Hom_{S}(\cK,
                                R) \ar[r] \ar[d] & \Hom_{S}(\cK, R_0) \ar[d] & \\
                            \ar[r] & \Ext^1_{S}(R_0, \cK) \ar[r] \ar[d] &
                            \Ext^1_{S}(R_0, R) \ar[r] \ar[d, "\varphi'"] &
                            \Ext^1_{S}(R_0, R_0) \ar[d, "\varphi"] & \\
                            \ar[r] & \Ext^1_{S}(R, \cK) \ar[r] \ar[d] &
                            \Ext^1_{S}(R, R) \ar[r, "b'"] \ar[d, "c'"] &
                            \Ext^1_{S}(R, R_0) \ar[d, "c"] & \\
                            \ar[r] & \Ext^1_{S}(\cK, \cK) \ar[r, "d"]  & \Ext^1_{S}(\cK,
                            R) \ar[r, "b"]  &
                            \Ext^1_{S}(\cK, R_0)
                    \end{tikzcd}
                    \caption{Long exact sequence of
                    $\Ext$-modules}\label{fig:les}
                \end{figure}
                We have
                \[
                    \degold \Ext^1_S(R, R) = \degold \im c' + \degold \im \varphi'.
                \]
                The map $bc'$ factors through $c$ and $\im c = \coker
                \varphi$, so we have
                \begin{equation}\label{eq:ineqbcISc}
                    \degold \im c' \leq \degold \im(bc') + \degold \im d \leq \degold
                    \im c  + \degold \im d = \degold \coker \varphi + \degold
                    \im d.
                \end{equation}
                We also have
                \begin{align*}
                    \degold \im d &= \degold \Ext^1_S(\cK, \cK) - \degold \Hom_S(\cK,
                    R_0) + \degold \Hom_S(\cK, R) - \degold \Hom_S(\cK, \cK), \text{ and }\\
                    \degold \im \varphi' &= \degold \Ext^1_S(R_0, R) - \degold \Hom_S(\cK,
                    R) + \degold \Hom_S(R, R) - \degold \Hom_S(R_0, R).
                \end{align*}
                By definition, we have
                $\degold \Ext^1_S(\cK, \cK) =
                3\degold \cK + \delta_\cK$. Since $\cK$ is cyclic or cocyclic, it follows that
                $\degold \Hom(\cK, \cK) = \degold \cK$. Since $R$ is cyclic, we also have
                $\degold \Hom_S(R, R) =
                \degold R = \degold \cK + \degold R_0$. By definition, $\Hom_S(\cK, R_0) = \ker \varphi$. Substituting
                these into the equations above, we obtain
                \begin{align}\label{eq:main}
                    \degold \Ext^1_S(R,R) &\leq \degold \im d + \degold \im \varphi' +
                    \degold \coker \varphi \\ \notag
                    & =  3\degold \cK + \delta_\cK + \degold \coker \varphi - \degold\ker \varphi + \\  
                    & \,\quad   \degold R_0 + \degold \Ext^1_S(R_0, R) - \degold \Hom_S(R_0, R).\notag
                \end{align}
                Now, \Cref{ref:transposition:lem} implies that $\sum_{i=0}^3
                (-1)^i\degold \Ext^i_S(R_0, R) = 0$, which in turn implies that
                \[
                    \degold \Ext^1_S(R_0, R) - \degold \Hom_S(R_0, R) = \degold
                    \Ext^2_S(R_0, R) - \degold \Ext^3_S(R_0, R).
                \]
                Applying Serre duality \cite[Lemma~2.2]{RS22} to the summands on the right-hand side, we get
                \begin{equation}\label{eq:second}
                    \degold \Ext^1_S(R_0, R) - \degold \Hom_S(R_0, R) = \degold
                    \Ext^1_S(R, R_0) - \degold \Hom_S(R, R_0).
                \end{equation}
                We have $\degold \Hom_S(R, R_0) = \degold R_0$ and $\degold
                \Ext^1_S(R, R_0) = \degold \Ext^1_S(R_0, R_0) - \degold \ker
                \varphi + \degold \coker \varphi$. Since $R_0$ is Gorenstein,
                we have $\degold \Ext^1_S(R_0, R_0) = 3\degold R_0$. Substituting
                these equalities into \Cref{eq:second} we obtain
                \begin{equation*}
                    \degold \Ext^1_S(R_0, R) - \degold \Hom_S(R_0, R) = 
                    2\degold R_0 - \degold \ker \varphi + \degold \coker \varphi.
                \end{equation*}
                Plugging this
                into \Cref{eq:main}, we obtain
                \[
                    \degold \Ext^1_S(R, R) \leq 3\degold R + \delta_\cK +
                    2\left(\degold \coker
                    \varphi - \degold \ker \varphi
                    \right).
                \]
                By \Cref{ref:cokernel:lem}, we have
                \[
                    \degold \coker \varphi - \degold \ker \varphi \leq \degold \Sym_R^2 \cK - \degold \cK,
                \]
                which concludes the proof of the inequality. 
If equality holds in the above equation, then all the inequalities in \Cref{eq:ineqbcISc} must be equalities. 
When the leftmost inequality in \Cref{eq:ineqbcISc} is an equality, it implies that $\im(c')$ contains $\im(d)$. 
If the second inequality in \Cref{eq:ineqbcISc} is an equality, then $\im(cb') = \im(c)$. 
Since $bc' = cb'$, it follows that $\im(b'c) = \im(c)$, and thus $\Ext^1_S(R, R_0) = \im(b') + \ker(c) = \im(b') + \im(\varphi)$.
            \end{proof}

            \begin{cor}\label{ref:smoothnessOfSmoothableBroken:cor}
                Let $S = \kk[x, y, z]$ and  $R$  a finite quotient
                algebra of $S$.
                Assume $R$ has a broken Gorenstein  structure,
                with first step given by the exact sequence
                $0\to \cK\to R\to R_0\to 0$.
Then, we have $\delta_R \leq 0$.

                If we also assume that $R$ is smoothable, then
                \begin{enumerate}
                    \item $[R]\in\Hilb(\AA^3)$ is a smooth point,
                    \item $[R\onto R_0]\in \Hilb^{d, d_0}(\AA^3)$ is a smooth
                        point of the nested Hilbert scheme, and
                    \item the map of abstract deformation functors
                        $\Def_{\cK \subseteq R}\to \Def_\cK$ is smooth.
                \end{enumerate}
            \end{cor}
            \begin{proof}
             By assumption,  $\cK$ is either cyclic or cocyclic.
If $\cK$ is cyclic, then   $\cK \simeq A = R/\Ann(\cK)$, thus, 
              $\Sym^2_R \cK \simeq \Sym^2_{A} A\simeq A$,
              so $\degold(\Sym^2_R \cK) = \degold \cK$.
If $\cK$ is cocyclic, applying
                \Cref{ref:brokenGorensteinCanonical:cor} to
                $R/\Ann(\cK)$, we obtain that $\degold(\Sym^2_R \cK) \leq \degold \cK$.
                Thus, \Cref{ref:canitbetrue:thm} and \Cref{ref:transposition:lem} imply that $\delta_R
                \leq \delta_{\cK}= \delta_{R/\Ann(\cK)}$. 
                By assumption, the algebra
                $R/\Ann(\cK)$ is also a quotient of $S$ and admits a broken Gorenstein
                structure with a smaller number of steps. Therefore, by induction, we
                have $\delta_{R/\Ann(\cK)}\leq 0$, which implies $\delta_{R}\leq 0$. Consequently, the dimension of the
                tangent space to $[R]$ is at most $3\degold R$.

Assume $R$ is smoothable. Since the tangent space at $[R]$ has dimension at most $3\degold R$, the point $[R]\in \Hilb(\AA^3)$ is smooth. Moreover, in
                this case, we have $\delta_R = 0$, while
                $\delta_{\cK} \leq 0$ and $\degold(\Sym^2_R \cK) \leq \degold
                \cK$. In particular, equality holds in~\Cref{eq:tangent_excessChange}. By
                \Cref{ref:canitbetrue:thm} \eqref{it:canitbetrue1} and \cite[Theorem~A.2]{jelisiejew_Mandziuk}, the nested Hilbert
                scheme is smooth at $[R\onto R_0]$. This implies
                that the abstract deformation functor $\Def_{R\onto R_0}$ is
                (formally) smooth. This functor is
                isomorphic to $\Def_{\cK\subseteq R}$, see
                \Cref{ssec:defFunct}. The forgetful functor
                $\pi\colon \Def_{\cK\subseteq
                R}\to \Def_{\cK}$ induces a map on tangent spaces
                $d\pi\colon T{\Def_{\cK\subseteq R}}\to T{\Def_{\cK}}$,
                where $T{\Def_{\cK}} =
                \Ext^1_S(\cK, \cK)$ and
                \[
                    T{\Def_{\cK \subseteq R}} = \left\{ (e_R, e_{\cK})\in
                \Ext^1_S(R, R) \oplus \Ext^1_S(\cK, \cK)\ |\ c'(e_R) =
                d(e_{\cK}) \right\},
                \]
                see again~\Cref{ssec:defFunct}.
                By \Cref{ref:canitbetrue:thm} \eqref{it:canitbetrue2}, the map $d\pi$ is
                surjective. Finally, the ``standard criterion for
                smoothness'' ~\cite[Lemma~6.1]{Fantechi_Manetti}, implies that the map $\pi$
                is (formally) smooth.
            \end{proof}

    \subsection{Broken Gorenstein algebras without flips are
    licci}\label{sec:licci}

The main result of this section is that broken Gorenstein algebras without
flips are licci. 
This supports \Cref{main_conjecture}.
The key point is a linkage lemma proven by Huneke-Polini-Ulrich~\cite{HPU}, which refines~\cite{Watanabe__A_note}. 
For the next two results, fix a polynomial ring $S = \kk[x,y,z]$.

\begin{prop}[Huneke, Polini, Ulrich {\cite{HPU}}]\label{PropSameEvenLinkageClass}
Let $I\subseteq S$ be a cofinite ideal and $f \in S$ be such that $I+(f)$ is Gorenstein.
Then, $I$ and $(I:f)$ are in the same (even) linkage class.
\end{prop}

\begin{thm}\label{ref:NoFlipsLicci:thm}
If $R = S/I$ is a finite algebra with a broken Gorenstein structure without flips, then the
ideal $I$ is licci.
\end{thm}
\begin{proof}
    By definition, we have an exact sequence of $S$-modules $0\to \cK\to R\to
    R_0\to 0$ with $R_0$ Gorenstein and $\cK  \simeq Rf$ cyclic. The map $S\to R$ sending $1$ to $f$ has kernel $(I:f)$, so $\cK
    \simeq S/(I:f)$.
    By \Cref{PropSameEvenLinkageClass}, the ideal $I$ defining $R$ is evenly linked to the ideal
    $(I:f)$ defining $\cK$. Proceeding by induction on the number of steps in the structure, $I$ is evenly linked to a Gorenstein ideal in $S$, the latter of which is known to be licci  \cite[proof of the Theorem]{Watanabe__A_note}.
\end{proof}

\begin{remark}
 The implication ``smooth
    $\implies$ licci'' in \Cref{main_conjecture} is  very particular to
    codimension $3$. The following is a counterexample in codimension four: the ideal defining
    $I = (x^2,xy,y^2) + (z^2,zw,w^2) \subseteq \kk[x,y,z,w]$ is not licci
    \cite[Theorem 2.6]{Huneke_Ulrich__Monomials}, but it is a smooth point of $\Hilb^9(\AA^4)$.  
\end{remark}

\subsection{Structure theorem in the case without flips}

In this section, we prove  \Cref{theorem_pfaffian_structure} and \Cref{Prop_pfaffian_converse}.

\begin{proof}[Proof of \Cref{theorem_pfaffian_structure}] Recall that $S$ is a regular local ring and let $R = S/I$.
    Since the
    broken Gorenstein structure has no flips, \Cref{eq:noflipsFlag} implies
    that there are elements $\alpha_1, \ldots ,\alpha_k\in S$ and a chain of principal ideals 
\[
    0 = I_{k+1}  \subsetneq I_{k} \subsetneq I_{k-1} \subsetneq \cdots \subsetneq I_1 \subseteq I_0 = R
\]
such that (putting $\alpha_0 = 1$, $\alpha_{k+1} = 0$) for every
$i=0,1, \ldots ,k$ we have $I_i = R\alpha_0\alpha_1\cdots\alpha_i$ and $I_{i}/I_{i+1}$
is isomorphic to a Gorenstein quotient of $R$.

Fix an $0\leq i\leq k$ and write $I_{i}/I_{i+1} \simeq
S/K_i$, then $S/K_i$ is Gorenstein of codimension three by assumptions, so,
by the 
Buchsbaum-Eisenbud theorem \cite{BuchsbaumEisenbudCodimThree}, there exist a skew-symmetric matrix $A_i$ and a minimal free resolution
\begin{align} \label{BE_i}
0 \to S \xrightarrow{\text{Pf}(A_i)^{T}} S^{n_i} \xrightarrow{A_i} S^{n_i} \xrightarrow{\text{Pf}(A_i)} S \to S/K_i \to 0.
\end{align}
By abuse of notation, we also use $\text{Pf}(A_i)$ to denote the $1 \times n_i$ row vector whose entries generate the corresponding pfaffian ideal (as defined in \Cref{broken_intro}). Consider the map $S\to I_i$ that sends $1$ to $\alpha_0\alpha_1 \cdots \alpha_i$.
It is surjective, since $I_i$ is generated by $\alpha_0\alpha_1 \cdots \alpha_i$. 
Let $L_i$ be
its kernel, so that $I_i  \simeq S/L_i$. Dividing $S/L_i$ by
$\alpha_{i+1}$ corresponds to dividing $I_i$ by the product $\alpha_0\alpha_1 \ldots
\alpha_{i+1}$ and so
\begin{equation}\label{eq:decomposition}
    K_i = L_i + (\alpha_{i+1}).
\end{equation}
Note that $\alpha_{i+1}$ is a minimal generator for $0 \leq i \leq k-1$. Indeed, if $\alpha_{i+1}$ is not a minimal generator of $K_i$, it follows from
Nakayama's lemma that $K_i = L_i$. Thus, $I_{i}/I_{i+1} \simeq S/K_i = S/L_i \simeq I_i$, i.e., $I_{i+1} = 0$, which is a contradiction for $i\leq k-1$.

Assume $0 \leq i \leq k-1$. Since $\alpha_{i+1}$ is a minimal generator of $K_i$, we can find an
invertible matrix $g\in M_{n_i\times n_i}(S)$ such that the first element
of the vector $g\Pf(A_i)$ is equal to $\alpha_{i+1}$. Replacing~\eqref{BE_i} by the resolution
with maps $g^{\intercal}\Pf(A_i)^{\intercal}$, $g^{-1}A_i
(g^{\intercal})^{-1}$, $\Pf(A_i)g$, we obtain another self-dual resolution and additionally
we get that $\alpha_{i+1} = \Pf(A_i)_1$. By~\eqref{eq:decomposition}, every
other Pfaffian of $A_i$ is a sum of a multiple of $\alpha_{i+1}$ and an
element of $L_i$.  By acting with another invertible matrix we can guarantee that the remaining Pfaffians lie in $L_i$, that is, that $\Pf(A_i)_{\geq 2}
\subseteq L_i$. Indeed, if we write the ideal $\Pf(A_i) = (p_1,\dots,p_{n_i})$ then $p_j = f_j\alpha_{i+1} + \ell_j$ for $\ell_j \in L_i$. Let $E_j$ be the upper triangular matrix with $1$s along the diagonal and $-f_j$ in the $(1,j)$-th entry. The desired invertible matrix is $E_2\circ \cdots \circ E_{n_i}$. By construction of $L_i$, this implies that
\begin{equation}\label{eq:otherElements}
    \alpha_0\alpha_1 \cdots \alpha_i \Pf(A_i)_{\geq 2} \subseteq I.
\end{equation}
Note that for $i=k$, we in fact have $\alpha_0\alpha_1 \cdots \alpha_k \Pf(A_k)_{\geq 1} \subseteq I$ since $K_i= L_i$.

Consider a surjection $S^{\oplus k+1}\to R$ given by
\[
    [\alpha_{k}\cdots \alpha_{0},\ \alpha_{k-1}\cdots \alpha_{0},\ \dots,\ \alpha_1\alpha_0,\ \alpha_0].
\]
Let $N = \sum_{i=0}^k n_i$ be the sum of sizes of $A_0, \ldots ,A_k$.
Let $C_{i}$ denote the vector $[-1, 0, 0, \ldots, 0]$ of length $n_i$, that
is, of size equal to the size of $A_i$.
Consider the map $S^{\oplus N} \to S^{\oplus k+1}$ given by the matrix
\[
M = 
\begin{pmatrix}
\Pf(A_{k}) & C_{k-1} & 0 &  \ldots & 0\\
0 & \Pf(A_{k-1}) & C_{k-2} &  & \vdots\\
\vdots &&  \ddots &&0\\
0&\cdots&0& \Pf(A_{1}) & C_{0}\\
0& \cdots &0&0& \Pf(A_0)
\end{pmatrix}
\]
Thanks to~\eqref{eq:otherElements}, we obtain a complex
            \begin{equation}\label{eq:complex}
                \begin{tikzcd}
                    0 & \ar[l] R & \ar[l, ] S^{\oplus k+1} & \ar[l, "M"']
                    S^{\oplus N}.
                \end{tikzcd}
            \end{equation}
            We will prove that it is exact.
        Consider the Rees-like $R[t]$-module $\mathcal{R} := I_kt^{-k} \oplus I_{k-1}t^{-(k-1)}\oplus \ldots
        \oplus I_1t^{-1} \oplus R \oplus Rt \oplus Rt^{2}\oplus  \ldots
        \subseteq R[t^{\pm 1}]$  associated to the
        filtration on $R$. Since $\mathcal{R}$ is a torsion-free $\kk[t]$-module, it is flat.
        The complex above generalizes to
            \begin{equation}\label{eq:complexExtended}
                \begin{tikzcd}
                    0 & \ar[l] \mathcal{R} & \ar[l] S[t]^{\oplus k+1} & \ar[l,
                    "M(t)"']
                    S[t]^{\oplus N}
                \end{tikzcd}
            \end{equation}
            where
            \[
                M(t) = 
\begin{pmatrix}
\Pf(A_{k}) & tC_{k-1} & 0 &  \ldots & 0\\
0 & \Pf(A_{k-1}) & tC_{k-2} &  & \vdots\\
\vdots &&  \ddots &&0\\
0&\cdots&0& \Pf(A_{1}) & tC_{0}\\
0& \cdots &0&0& \Pf(A_0)
\end{pmatrix}
            \]
            and the surjection is $ [\alpha_{k}\cdots \alpha_{0}t^{-k},\ \alpha_{k-1}\cdots \alpha_{0}t^{-(k-1)},\ \dots,\ \alpha_1\alpha_0t^{-1},\ \alpha_0]$.
            
            Let $\mathcal{I} = \ker(S[t]^{\oplus k+1}\to \mathcal{R})$. Since the
            $\kk[t]$-module $\mathcal{R}$ is flat, for every $\lambda\in \kk$
            the module $\mathcal{I}/(t-\lambda)\mathcal{I}$ is the kernel of
            $S[t]/(t-\lambda)^{\oplus k+1} \to
            \mathcal{R}/(t-\lambda)\mathcal{R}$. In particular, the homology
            in the middle of the complex (\ref{eq:complexExtended}) commutes
            with base change of $\kk[t]$.
            Moreover, the complex (\ref{eq:complexExtended}) is $\Z$-graded and becomes exact after dividing by $(t)$ since $\alpha_{i+1}\alpha_i \alpha_{i-1}\cdots \alpha_0 t^{-i} = (\alpha_{i+1}\alpha_i\cdots \alpha_0t^{-(i+1)})t \in t\cdot I_{i+1}t^{-(i+1)}$.
            Consider the homomorphism of $\kk[t]$-modules $S[t]^{\oplus N}\to \mathcal{I}$. Since this map stays surjective after
            dividing by $t$, the map must be surjective.
            Applying Nakayama's lemma we conclude that the complex (\ref{eq:complexExtended}) is exact.
            Since its homology in the middle commutes with base change, it
            stays exact after dividing by $(t-1)$. After this division we
            obtain the complex (\ref{eq:complex}), in particular, this complex is also exact.
            It is straightforward to read off the generators of $I$ from $M$.
            \end{proof}
            
            \begin{proof}[Proof of \Cref{Prop_pfaffian_converse}]
Given such a collection of skew-symmetric matrices $A_0, \ldots, A_k$, we can define $\alpha_{i+1} = \Pf(A_i)_1$, $I$ as in \Cref{eq:brokenExplicitGens}, $R = S/I$, and $I_i = R\alpha_0\alpha_1\cdots\alpha_{i}$ with $\alpha_0 =1$ and $\alpha_{k+1}=0$. This gives us a chain of principal ideals $ 0 = I_{k+1}  \subsetneq I_{k} \subsetneq I_{k-1} \subsetneq \cdots \subsetneq I_1 \subseteq I_0 = R$. It remains to show that $I_i/I_{i+1} $ is Gorenstein.
           Observe that,
          $$\frac{I_i}{I_{i+1}} 
          = \frac{\sum_{j=0}^{i-1} \alpha_0\cdots\alpha_j\Pf(A_j)_{\geq 2} + (\alpha_0\cdots\alpha_{i})}
          		{\sum_{j=0}^{i} \alpha_0\cdots\alpha_j\Pf(A_j)_{\geq 2} + (\alpha_0\cdots\alpha_{i+1})}
          = \frac{(\alpha_0\cdots\alpha_{i})}{\alpha_0\cdots\alpha_{i}\Pf(A_{i})_{\geq 2}+ (\alpha_0\cdots\alpha_{i+1})}
          $$
          with the last equality following from \Cref{pfaffian_condition}.
This simplifies to 
$$
\frac{(\alpha_0\cdots\alpha_{i})}{\alpha_0\cdots\alpha_{i}\Pf(A_{i})_{\geq 2}+ (\alpha_0\cdots\alpha_{i+1})}
\cong \frac{S}{\Pf(A_{i})_{\geq 2}+ \alpha_{i+1}} = \frac{S}{\Pf(A_{i})},
$$
as required.
\end{proof}

\section{Smooth monomial points} \label{sec:mon1}
\setlength{\epigraphwidth}{0.4\textwidth}

In this section, we develop an explicit criterion for determining the
smoothness of a monomial point in $\Hilb^d(\AA^3)$, see \Cref{ThmSmoothMonomialClassification}. 
We also prove that singular monomial points have smoothable tangent excess greater than or equal to 6, see  	\Cref{ThmSingularMonomial}.
In particular, we prove \Cref{main_conjecture} and \Cref{conj:hu1}
for monomial ideals.

Throughout this section, the polynomial ring $S$ is always  $S=\kk[x,y,z]$.
We denote elements of $\Z^3$ by bold-face letters such as $\bfa, \bfb$
and sometimes endow them with superscripts to enumerate them, for example, $\bfc^1,\bfc^2$ etc.
Given an element $\bfa \in \Z^3$, we use $\bfa_i$ to denote the $i$-th component of the vector $\bfa$, that is, $\bfa = (\bfa_1, \bfa_2, \bfa_3)\in\Z^3$.
By an abuse of notation, 
we identify a monomial $x^{\bfa_1}y^{\bfa_2}z^{\bfa_3}$
with its exponent vector $\bfa = (\bfa_1, \bfa_2, \bfa_3)$.
We identify a monomial ideal $I \subseteq S$ with the set of exponents vectors $I\subseteq \N^3$, 
the monomials of $S/I$ with the staircase $E_I := \N^3 \setminus I$.
Throughout this section, we denote by $\bfa \leq \bfb$ the partial order in $\Z^3$ (or $\Z^2$) given by componentwise inequality, equivalently, by divisibility of the corresponding monomials.

\subsection{Monomial points with no singularizing triples} 
We begin by describing the structure of monomial ideals without singularizing triples.

\begin{prop} \label{PropNoSingularizingTriple}
Let $I\subseteq S = \Bbbk[x,y,z]$ be a cofinite monomial ideal.
The following conditions are equivalent:
\begin{enumerate}[label=(\roman*)]
\item $I$ admits no singularizing triple.
\item For every subset $\mathcal{E}\subseteq \soc(S/I)$,
    there exists $\bfs \in \mathcal{E}$ and two indices $i, j\in \{1,2,3\}$ such that
$$
\bfs_i = \max(\bft_i\ |\ \bft\in \mathcal{E}),\ \ \mbox{and}\ \
\bfs_j = \max(\bft_j\ |\ \bft\in \mathcal{E})
$$
We also have $\bfs_k < \bft_k$ for all $\bft \in \mathcal{E}\setminus\{\bfs\}$,
where $\{i,j,k\}=\{1,2,3\}$.
\item There exists an ordering of the socle monomials
\begin{equation}\label{EqSocleOrdering}
\soc(S/I) = \big\{ \bfs^1, \bfs^2, \ldots, \bfs^\tau\big\}
\end{equation}
such that, for every $p$, the  monomial $\bfs^p$ dominates the subsequent monomials in two components $i_p, j_p \in \{1,2,3\}$ depending on $p$:
\begin{equation}\label{EqSocleDomination}
\bfs^p_{i_p} \geq \bfs^q_{i_p},
\,
\bfs^p_{j_p} \geq \bfs^q_{i_p}
\quad
\forall\, q >p.
\end{equation}
Moreover, we also have $\bfs^p_{k_p} < \bfs^q_{k_q}$ for all $q>p$,
where $\{i_p,j_p,k_p\}=\{1,2,3\}$.
\end{enumerate}
\end{prop}

\begin{proof} We prove (i) $\Rightarrow$ (iii) $\Rightarrow$ (ii) $\Rightarrow$ (i). 
\begin{enumerate}[itemindent=15pt]
\item[(i) $\Rightarrow$ (iii)] For each $i \in \{1,2,3\}$ consider $m^1_i =  \max\{\bfa_i \,\mid \, \bfa \in \soc(S/I)\}$. Since $I$ has no singularizing triples, there exists $\bfs^1 \in \soc(S/I)$ such that $\bfs^1_i = m^1_i$ for at least two $i\in \{1,2,3\}$. If this was not the case, then  the three monomials attaining the maxima $m^1_1, m^1_2, m^1_3$ would form a singularizing triple. To construct the next element, we consider $m^2_i =  \max\{\bfa_i \,\mid \, \bfa \in \soc(S/I)\setminus\{\bfs^1\}\}$ and similary choose $\bfs^2$ to be the element attaining $m^2_i$ for at least two $i \in \{1,2,3\}$. Repeatedly applying this procedure gives us the ordering in \Cref{EqSocleOrdering} and, by construction, it satisfies \Cref{EqSocleDomination}. The last statement follows from the incomparability of monomials in $\soc(S/I)$ with respect to the partial order given by divisibility.
\item[(iii) $\Rightarrow$ (ii)] This follows immediately by restricting the order to $\mathcal{E}$.
\item[(ii) $\Rightarrow$ (i)] This follows by applying (ii) to every triple $\mathcal{E}=\{\bfa, \bfb, \bfc\} \subseteq \soc(S/I)$. \qedhere
\end{enumerate}
\end{proof}

\begin{remark}\label{RemarkProduceSmooth}
Using \Cref{PropNoSingularizingTriple}, one can effectively  generate all the 
monomial ideals $I$  that admit no singularizing triples, up to a given value of  $\dim_\kk S/I$.
Indeed, 
such monomial ideals are encoded by a sequence 
$\big\{ \bfs^1, \bfs^2, \ldots, \bfs^\tau\big\}$
satisfying 
 condition (iii). 
By the nature of this condition, all such sequences can be generated by a simple recursive procedure in $\tau$.
This observation allows to compute  the generating series $P^{\sm}_3(d)$ introduced in 
\Cref{subsecGenFunc} up to a given finite order.
\end{remark}

\begin{prop} \label{PropSmoothMonomial}
Let $I\subseteq S = \Bbbk[x,y,z]$ be a cofinite monomial ideal.
If $I$ admits no singularizing triple, then $S/I$ has a broken Gorenstein
structure without flips,
in particular,  $[S/I] \in \Hilb^d(\AA^3)$ is a smooth point.
 \end{prop}

 \begin{proof}
To prove that a broken Gorenstein structure without flips exists, we use induction on the dimension $\tau$ of the socle of $S/I$. 
The case $\tau=1$ is trivial, so assume $\tau >1$.

Order the socle elements $\bfs^1, \ldots ,\bfs^{\tau}$ of $R = S/I$ as in  \Cref{EqSocleOrdering}. 
Let $\bfs^1 = (\bfs^1_1, \bfs^1_2, \bfs^1_3)$ be the first socle element and assume that its coordinates $(-)_1$ and $(-)_2$ dominate the others. 
It follows that $\bfs^1_3 < \bfs^i_3$ for every $i=2,3, \ldots , \tau$.

Let $f := z^{\bfs^1_3 + 1}$.
The canonical module $\omega_{R/(f)}\subseteq \omega_{R}$ consists of elements annihilated by $f$, so the only socle element of $R/(f)$ is $\bfs^1$ and $R/(f)$ is  Gorenstein.

The principal ideal $Rf$ is isomorphic, as $R$-module,  to the algebra $S/(I:f)$.
The monomial basis of $Rf$ can be identified with the monomials in $R$  divisible by $f$. This shows that socle elements of $S/(I : f)$ are $f^{-1}\bfs^2, \ldots , f^{-1}\bfs^{\tau}$. By the criterion \Cref{PropNoSingularizingTriple}, also $(I:f)$ admits no singularizing triple. By induction there is a broken Gorenstein structure without flips on $S/(I:f) \simeq Rf$. 
Merging it with $0\to Rf \to R\to R/Rf\to 0$, we obtain a broken Gorenstein structure without flips on $S/I$. 
The smoothness of $S/I$ follows from \Cref{ref:smoothnessOfSmoothableBroken:cor}.
 \end{proof}

\subsection{Monomial points with singularizing triples}

In this subsection, we
show that the monomial points with singularizing triples are singular points on the Hilbert scheme. 

We begin by recalling a criterion for smoothness from~\cite{RS22} involving the weights of the tangent vectors.

\begin{definition}\label{DefinitionSubspaces}
A \textbf{signature} is a non-constant triple on the two-element set  $\{\text{p},\text{n}\}$, where
$\text{p}$ stands for ``positive or 0'' while $\text{n}$ stands for ``negative''. 
Let $\mathfrak{S}=\{\ppn,\pnp,\npp,\nnp,\npn, \pnn\}$ denote the set of signatures, and 
for each $\mathfrak{s} \in \mathfrak{S}$ define
$
\Z^3_{\mathfrak{s}} 
	= \big\{\bfa  \in \Z^3: \bfa_i \geq 0 \, \text{ if } \, \mathfrak{s}_i = \p,  \,\text{and }  \bfa_i < 0 \,\, \text{ if } \, \mathfrak{s}_i = \n \big\}.
$
\end{definition}

Given a monomial ideal $I \subseteq S$, we define the subspaces
$$
T_\mathfrak{s}(I) =  \bigoplus_{\bfa \in \Z^3_{\mathfrak{s}}} T(I)_{\bfa} \subseteq T(I)
$$
where $ T(I)_{\bfa}$ denotes the graded component of $T(I)$ of degree $\bfa\in\Z^3$. 
It can be shown that $T_{\text{ppp}}(I)=T_{\text{nnn}}(I)=0$, and therefore $T(I) = \bigoplus_{\mathfrak{s} \in \mathfrak{S}} T_{\mathfrak{s}}(I)$ \cite[Proposition 1.9]{RS22}.

\begin{definition}\label{ref:doublyNegative:def} Let $I \subseteq S$ be a monomial ideal. We say that a vector $\bfa \in \Z^3$ is \textbf{doubly-negative} if $\bfa \in \Z^3_{\nnp} \cup \Z^3_{\npn} \cup \Z^3_{\pnn}$. Similarly, a non-zero tangent vector $\varphi \in T(I)$ is said to be \textbf{doubly-negative} if $\varphi \in T_{\nnp}(I) \cup T_{\npn}(I) \cup T_{\pnn}(I)$.
\end{definition}

We now state the two results from \cite{RS22} that we will need.
\begin{prop}[{\cite[Theorem 2.4]{RS22}}] \label{theorem_smooth_monomials} 
Let $I \subseteq S$ be a monomial ideal such that $\dim_\kk(S/I) = d$.
Then,
$$
\dim_{\kk} T_{\ppn}(I) = \dim_{\kk} T_{\nnp}(I) + d, \quad 
\dim_{\kk} T_{\pnp}(I) =  \dim_{\kk} T_{\npn}(I) + d, \quad \text{and }
\dim_{\kk} T_{\npp}(I) = \dim_{\kk} T_{\pnn}(I) + d. 
$$
In particular, the point $[S/I] \in \Hilb^d(\AA^3)$ is smooth if and only if $T_{\nnp}(I) = T_{\npn}(I) = T_{\pnn}(I)=0$.
\end{prop}

We can now formulate our main result,
which settles the monomial case of \cite[Conjecture 4.25]{H23}.

\begin{thm} \label{ThmSingularMonomial} Let $I \subseteq S$ be a cofinite
    monomial ideal. If $I$ admits a singularizing triple, then $T_{\nnp}(I)$,
    $T_{\npn}(I)$ and $T_{\pnn}(I)$ are all non-zero, so $\degold T(I)\geq
    3\degold(S/I) + 6$.
\end{thm}
\begin{proof}
    Once we show that $T_{\nnp}(I)$, $T_{\npn}(I)$, $T_{\pnn}(I)$ are
    non-zero, the remaining part of the claim follows from
    \Cref{theorem_smooth_monomials}.

Let  $\{\bfa, \bfb, \bfc\} \subseteq \soc(S/I)$ be a singularizing triple, with 
\[
\bfa_1 >\bfb_1, \bfc_1,
\quad
\bfb_2 > \bfa_2, \bfc_2,
\quad
\bfc_3 > \bfa_3, \bfb_3.
\]

Up to replacing $\bfc$, we assume that the third coordinate $\bfc_3$ is the largest  among all the socle elements $\bfc$ satisfying $\bfc_1< \bfa_1$ and  $\bfc_2 < \bfb_2$.
We are going to  construct a bounded connected component corresponding to the
doubly-negative signature $\nnp$ (\Cref{proposition_basis}). In particular,
this will ensure $T_{\nnp}(I) \ne 0$, and by symmetry, our argument will also imply that $T_{\npn}(I), T_{\pnn}(I) \ne 0$.

For each $k \in \N$, we define the \textbf{$k$-th levels} of $I$ and $E$ by
$$
I_k = \big\{ (v_1, v_2) \in \N^2 \, \mid \, (v_1, v_2,k) \in I \big\},
\qquad
E_k = \big\{ (v_1, v_2) \in \N^2 \, \mid \, (v_1, v_2,k) \in E_I \big\} = \N^2 \setminus I_k.
$$ 
In particular, we may interpret $I_k$ as an ideal in $\kk[x,y]$. 
We set $I_{-1} := \emptyset$ and we have the containments $I_{k-1} \subseteq I_{k}$ for all $k\geq 0$.

\underline{Low level:}
Let $\ell \in \N$ be the smallest integer such that $(\bfa_1, \bfb_2, \ell) \in I$.
Since $\bfa, \bfb \in \soc(S/I)$, we have that $(\bfa_1, \bfa_2+1, \bfa_3), (\bfb_1+1, \bfb_2, \bfb_3) \in I$, so 
$(\bfa_1, \bfb_2, \bfa_3), (\bfa_1, \bfb_2, \bfb_3) \in I$,
and therefore
 $0 \leq\ell \leq \min(\bfa_3,\bfb_3) < \bfc_3$.
Since $\bfa, \bfb \notin I$, we have $(\bfa_1, \bfa_2, \ell), (\bfb_1, \bfb_2, \ell) \notin I$.
We deduce that
\begin{equation}\label{EqLowAvoidsContour}
\big\{
\bfv \in \N^2\, \mid \, 
\bfv \leq (\bfa_1,\bfa_2) \text{ or }
\bfv \leq (\bfb_1,\bfb_2)
\big\}
\cap I_\ell = \emptyset.
\end{equation}
Moreover, by definition $(\bfa_1,\bfb_2, \ell-1) \notin I$, thus,
\begin{equation}\label{EqLevelBelowLow}
\big\{
\bfv \in \N^2\, \mid \, 
\bfv \leq (\bfa_1,\bfb_2)
\big\}
\cap I_{\ell-1} = \emptyset.
\end{equation}

\underline{High level:} Let $h = \bfc_3$.
We claim that
\begin{equation}\label{EqHighIncludesContour}
\big\{
\bfv \in \N^2\, \mid \, 
\bfv \geq (\bfa_1,\min\{\bfa_2,\bfc_2\}) \text{ or }
\bfv \geq (\min\{\bfc_1,\bfb_1\},\bfb_2)
\big\}
\subseteq I_h.
\end{equation}
Equivalently, it suffices to show that 
$(\bfa_1,\min\{\bfa_2,\bfc_2\})\in I_h$ and $
(\min\{\bfc_1,\bfb_1\},\bfb_2)	\in I_h$.
For the former:
if $\bfa_2 \leq \bfc_2$ then $(\bfa_1,\min\{\bfa_2,\bfc_2\})=
(\bfa_1,\bfa_2)\in I_h$ because $(\bfa_1,\bfa_2,h)\geq (\bfa_1,\bfa_2,\bfa_3+1)\in I$;
if $\bfa_2 > \bfc_2$ then $(\bfa_1,\min\{\bfa_2,\bfc_2\})=
(\bfa_1,\bfc_2)\in I_h$ because $(\bfa_1,\bfc_2,h)\geq (\bfc_1+1,\bfc_2,\bfc_3)\in I$.
The other is analogous.

Next, we claim that
\begin{equation}\label{EqLevelAboveHigh}
\big\{
\bfv \in \N^2\, \mid \, 
\bfv \geq (\min\{\bfc_1,\bfb_1\}, \min\{\bfc_2,\bfa_2\})\,\,\big\}
\subseteq I_{h+1}.
\end{equation}
Equivalently, it suffices to show that $(\min\{\bfc_1,\bfb_1\}, \min\{\bfc_2,\bfa_2\}) \in I_{h+1}$.
Assume by contradiction that $(\min\{\bfc_1,\bfb_1\}, \min\{\bfc_2,\bfa_2\}) \in E_{h+1}$.
Since $\soc(S/I)$ is the set of the maximal elements of $E_I$,
there exists $\bfw \in \soc(S/I)$ such that 
$
\bfw_1 \geq \min\{\bfc_1,\bfb_1\},
\bfw_2 \geq \min\{\bfc_2,\bfa_2\},
\bfw_3 \geq h+1.
$
If $\bfw_1 \geq \bfc_1 $ and $\bfw_2 \geq \bfc_2$ then $\bfw \geq \bfc$, contradicting
the fact that distinct socle monomials are incomparable.
Without loss of generality, 
we may assume that $\bfw_1 < \bfc_1$, and thus $\bfw_1 \geq \bfb_1$.
If $\bfw_2 \geq \bfb_2$, then $\bfw \geq \bfb$, which again gives a contradiction.
Thus, $\bfw_2 < \bfb_2$ and, since $\bfw_1 < \bfc_1 < \bfa_1$,  we conclude that $\{\bfa, \bfb, \bfw\}$ is a singularizing triple. 
Since $\bfw_3 \geq h+1 > \bfc_3$
this contradicts our choice of $\bfc$, and thus \Cref{EqLevelAboveHigh} is proved.

We will now use the low levels $E_\ell, I_\ell$ and the high levels $E_h, I_h$ to locate a bounded connected component of  $(I+\bfd) \setminus I$ for some appropriate vector $\bfd\in\Z^3_{\nnp}$.

Consider the rectangle 
$$
\mathcal{R}= \{\bfv \in \N^2 \, \mid \, 
\big(\min\{\bfc_1,\bfb_1\}, \min\{\bfc_2,\bfa_2\}\big) \leq \bfv \leq (\bfa_1,\bfb_2)\big\} \subseteq \N^2.
$$
Our discussion above implies a few things for the ideals $I_\ell, I_h$ at the low and high level, in relation to this rectangle. 
By definition of $\ell$, the ideal $I_\ell$ contains the upper-right corner of $\mathcal{R}$, i.e, 
$ (\bfa_1,\bfb_2) \in I_\ell$.
Moreover, by \Cref{EqLowAvoidsContour}, $I_\ell$ does not contain the lower-left perimeter of $\mathcal{R}$, i.e., 
$$
\big\{ \bfv \in \mathcal{R}  \mid  \bfv_1 = \min\{\bfc_1,\bfb_1\}
\text{ or } \bfv_2 =  \min\{\bfc_2,\bfa_2\}\,\, \big\}\cap I_\ell = \emptyset.
$$
The situation for the ideal $I_h$ at the high level is exactly mirrored.
Since $(\bfc_1, \bfc_2) \in E_h$,
the ideal $I_h$ does not contain the lower-left corner of $\mathcal{R}$, i.e., 
$(\min\{\bfc_1,\bfb_1\}, \min\{\bfc_2,\bfa_2\}) \notin I_h$. 
Moreover,  by \Cref{EqHighIncludesContour}, the ideal $I_h$ contains the upper-right perimeter of the rectangle of $\mathcal{R}$:
$$
\big\{
\bfv \in \mathcal{R} \, \mid\, 
\bfv_1 = \bfa_1
\text{ or }
\bfv_2 =  \bfb_2 \,
\big\}\subseteq I_h .
$$
All of this, together with the fact that $I_\ell \subseteq I_h$, shows that the assumptions of \Cref{LemmaDoubleNegativePlane} are satisfied.
Hence, there exist a $\bfd \in \Z^2$ with $\bfd \leq (-1,-1)$ and a (non-empty) connected component
$\mathcal{C}$ of the set
$\big((I_\ell\cap \mathcal{R})+\bfd\big)\cap \mathcal{R}\setminus I_h$
that is also a connected component of $(I_\ell+\bfd) \setminus I_h$.

Consider then the nnp vector $\bfe = (\bfd_1, \bfd_2, h-\ell)\in \Z^3$.
We have 
\begin{align*}
\mathcal{D} :=
\mathcal{C}\times \{h\} &\subseteq \big((I_\ell+\bfd) \setminus I_h\big)\times \{h\}\\
&=\big( (I_\ell + \bfd)\times \{h\}\big)\setminus (I_h\times\{h\}) \\
&=\big( (I+\bfe) \setminus I\big) \cap \{\bfv \in \Z^3 \,\mid \, \bfv_3 = h\}.
\end{align*}
We claim that $\mathcal{D}$ is a bounded connected component of $ (I+\bfe) \setminus I$;
this would conclude the proof of the theorem.
Clearly it is bounded and connected.
To show it is a  component, we must show that none of the points immediately  adjacent to $\mathcal{D}$ lie in $(I+\bfe) \setminus I$.
Note that such points $\bfv$ must have $\bfv_3 \in \{h-1,h,h+1\}$.
 We consider the three cases separately.

\begin{itemize}
\item Level $h$: 
Clearly, $\mathcal{D}$ is a connected component of $ \big( (I+\bfe) \setminus I\big) \cap \{\bfv \in \Z^3 \,\mid \, \bfv_3 = h\}.
$
This implies that no points of $\{\bfv \in \Z^3 \,\mid \, \bfv_3 = h\}$ that are immediately adjacent to $\mathcal{D}$ lie in $(I+\bfe) \setminus I$.
\item Level $h+1$:
The set of points of  $\{\bfv \in \Z^3 \,\mid \, \bfv_3 = h+1\}$ that are immediately adjacent to $\mathcal{D}$
is $\mathcal{C}\times\{h+1\}$.
Since $\mathcal{C}\subseteq \mathcal{R}$, and $\mathcal{R}\subseteq I_{h+1}$ by \Cref{EqLevelAboveHigh},
we have 
$$
(\mathcal{C}\times\{h+1\})\cap \big((I+\bfe)\setminus I\big) \subseteq 
(\mathcal{C}\times\{h+1\}) \setminus I = (\mathcal{C} \setminus I_{h+1})\times\{h+1\}=\emptyset.
$$
Thus, no points of $\{\bfv \in \Z^3 \,\mid \, \bfv_3 = h+1\}$ that are immediately adjacent to $\mathcal{D}$ lie in $(I+\bfd) \setminus I$.
\item Level $h-1$:
The set of points of  $\{\bfv \in \Z^3 \,\mid \, \bfv_3 = h-1\}$ that are immediately adjacent to $\mathcal{D}$
is $\mathcal{C}\times\{h-1\}$.
By \Cref{EqLevelBelowLow}, we see that $\mathcal{R}\cap I_{\ell-1} = \emptyset$.
We also have $\mathcal{C}\subseteq (I_\ell\cap \mathcal{R})+\bfd$.
Thus, 
\begin{align*}
(\mathcal{C}\times\{h-1\})\cap \big((I+\bfe)\setminus I\big) 
&\subseteq 
\big((I_\ell\cap \mathcal{R})+\bfd)\times\{h-1\}\big)\cap \big((I+\bfe)\setminus I\big)
\\
&\subseteq 
\big((I_\ell\cap \mathcal{R})+\bfd)\times\{h-1\}\big)\cap (I+\bfe)
\\
&=
\big((I_\ell\cap \mathcal{R})+\bfd)\times\{h-1\}\big)\cap (I+\bfe)\cap \{\bfv \in \Z^3 \,\mid \, \bfv_3 = h-1\}
\\
&
=
\big((I_\ell\cap \mathcal{R})+\bfd)\times\{h-1\}\big)\cap \big((I_{\ell-1}+\bfd)\times\{h-1\}\big)
\\
&
=
\big(((I_\ell\cap \mathcal{R})+\bfd)\cap (I_{\ell-1}+\bfd)\big)\times\{h-1\}
\\
&
=
\big((I_\ell\cap \mathcal{R}\cap I_{\ell-1})+\bfd\big)\times\{h-1\}
\\&
=
\big(( \mathcal{R}\cap I_{\ell-1})+\bfd\big)\times\{h-1\} = \emptyset.
\end{align*}
Again, no  point of $\{\bfv \in \Z^3 \,\mid \, \bfv_3 = h-1\}$ that is immediately adjacent to $\mathcal{D}$ lies in $(I+\bfd) \setminus I$. \qedhere
\end{itemize}
\end{proof}

\begin{lemma}\label{LemmaDoubleNegativePlane}
Let $L \subseteq H\subseteq \N^2$ correspond to two cofinite monomial ideals of 
$\Bbbk[x,y]$.
Let $\bfg, \bfs \in \N^2$ be two monomials with $\bfg_1 < \bfs_1$ and $\bfg_2 < \bfs_2$ and consider the rectangle $$
\mathcal{R}=\{ \bfv \in \N^2 \, \mid \, \bfg \leq \bfv \leq \bfs\}\subseteq \N^2.
$$
Suppose the following conditions hold:

\begin{itemize}
\item $L$ contains the upper-right corner of $\mathcal{R}$: $\bfs\in L$.
\item $L$ avoids the lower-left perimeter of $\mathcal{R}$: let $\mathcal{L}:= \{\bfv \in \mathcal{R}\,\mid\, \bfv_1 = \bfg_1 \text{ or } \bfv_2 = \bfg_2\}$, then $\mathcal{L}\cap L = \emptyset$.
\item $H$ does not contain the lower-left corner: $\bfg \notin H$.
\item $H$ contains the upper-right perimeter of $\mathcal{R}$:
let $\mathcal{U}:= \{\bfv \in \mathcal{R}\,\mid\, \bfv_1 = \bfs_1 \text{ or } \bfv_2 = \bfs_2\}$, then $\mathcal{U}\subseteq H$.
\end{itemize}
Then, there exists a $\bfd \in \Z^2$ such that $\bfd \leq (-1,-1)$ and a (non-empty) connected component
$\mathcal{C}$ of 
$$\big((L\cap \mathcal{R})+\bfd\big)\cap \mathcal{R}\setminus H$$
that is also a connected component of $(L+\bfd) \setminus H$.
\end{lemma}

\begin{proof}

Choose $\bfd$ to be a maximal
(with respect to the partial order $\leq $ in $\Z^2$)
element in the set
$$
\big\{
\bfd \in \Z^2\, \mid \, \bfd \leq (-1,-1) \text{ and } \big((L \cap \mathcal{R})+ \bfd\big) \cap  \mathcal{R} \setminus H\ne \emptyset
\big\}.
$$
This set is non-empty, since it contains the difference of the two corners $\bfg-\bfs$,
thus, $\bfd$ is well defined. 
Let $\mathcal{C}$ be a connected component of 
$\big((L\cap \mathcal{R})+\bfd\big)\cap \mathcal{R}\setminus H$.
We claim that  $\mathcal{C}$ is also a connected component of  $(L + \bfd) \setminus H$.
Assume by contradiction that this is not the case; then, 
there exist two adjacent points  $\bfp, \bfq \in (L + \bfd) \setminus H$
such that  $\bfp \in \mathcal{C}$ and $\bfq \notin \big((L\cap \mathcal{R})+\bfd\big)\cap \mathcal{R}\setminus H$.
More precisely, we have that either 
$\bfq \notin \mathcal{R}$
or 
$\bfq \in  \mathcal{R}\setminus \big((L \cap \mathcal{R})+\bfd\big)$.

\underline{Case  $\bfq \notin \mathcal{R}$.}
In this case, $\bfp$ must lie on the perimeter $\mathcal{L}\cup \mathcal{U}$ of the rectangle $\mathcal{R}$, and $\bfq$ must be adjacent to it, immediately outside $\mathcal{R}$.
Since $H$ contains the upper-right perimeter $\mathcal{U}$ and $\bfp \notin H$,
it follows that $\bfp\in\mathcal{L}\setminus\mathcal{U}$.
Thus, there are two further subcases: either
$\bfp$ lies in the bottom edge of $\mathcal{R}$ and $\bfq$ immediately below it, or
$\bfp$ lies in the left edge of $\mathcal{R}$ and $\bfq$ immediately to its left.
In terms of coordinates,
either
\begin{equation}\label{EqCaseInTheBottomEdge}
\bfg_1 \leq \bfp_1 < \bfs_1, \quad  \bfp_2 = \bfg_2,\quad 
\bfq_1 = \bfp_1, \quad
\bfq_2 = \bfp_2-1,
\end{equation}
or
\begin{equation}\label{EqCaseInTheLeftEdge}
\bfp_1 = \bfg_1,\quad  \bfg_2\leq \bfp_2 < \bfs_2,\quad
\bfq_1 = \bfp_1-1, \quad \bfq_2 = \bfp_2.
\end{equation}
Since the two cases are symmetric, we may, without loss of generality, assume that \Cref{EqCaseInTheBottomEdge} holds.

Since $\bfp \in \mathcal{C}\subseteq (L \cap\mathcal{R})+\bfd$, 
we have $\bfp-\bfd \in L \cap \mathcal{R} \subseteq  \mathcal{R}\setminus \mathcal{L}$, 
thus,
$\bfg_1 <\bfp_1-\bfd_1 \leq \bfs_1$ and $\bfg_2 < \bfp_2-\bfd_2 \leq \bfs_2$.
It follows from \Cref{EqCaseInTheBottomEdge}
that 
$\bfg_1 <\bfq_1-\bfd_1 \leq \bfs_1$ and $\bfg_2 \leq \bfq_2-\bfd_2 < \bfs_2$,
in particular, that $\bfq - \bfd \in \mathcal{R}$.
Since $\bfq \in L + \bfd$, we deduce that $\bfq -\bfd\in L\cap \mathcal{R} \subseteq \mathcal{R}\setminus \mathcal{L}$,
so $\bfq_2-\bfd_2 > \bfg_2$.
Since $\bfq_2 = \bfp_2-1=\bfg_2-1$,
we finally conclude that $\bfd_2<-1$.

Let $\bfd' = \bfd+(0,1)$. 
Observe that $\bfd'\leq (-1,-1)$ by the previous paragraph.
Since $\bfp-\bfd'=\bfq-\bfd$,
we have $\bfp-\bfd' \in L \cap \mathcal{R}$ by the previous paragraph.
Since $\bfp \in \mathcal{C}\subseteq \mathcal{R}\setminus H$, we have
$\bfp\in 
\big((L\cap \mathcal{R})+\bfd'\big)\cap \mathcal{R}\setminus H$,
contradicting  the maximality of $\bfd$.
This concludes the proof of this case.

\underline{Case $\bfq \in  \mathcal{R}\setminus \big((L \cap \mathcal{R})+\bfd\big)$.}
This case is analogous to the previous one with, roughly speaking, 
all the roles being reversed by the translation by $-\bfd$. 
The assumption is equivalent to $\bfq-\bfd \in  (\mathcal{R}-\bfd)\setminus (L \cap \mathcal{R})$.
Since $\bfq \in L+\bfd$, we have $\bfq-\bfd \in L$, and thus
$\bfq-\bfd \notin  \mathcal{R}$.
Since $\bfp \in \mathcal{R}+\bfd$, we have $\bfp-\bfd\in \mathcal{R}$.
Since $\bfp-\bfd,\bfq-\bfd$ are adjacent, 
we conclude that $\bfp-\bfd$ must lie on the perimeter $\mathcal{L}\cup \mathcal{U}$ of $\mathcal{R}$, 
and $\bfq -\bfd$ immediately outside $\mathcal{R}$.
Since $L\cap\mathcal{L}=\emptyset$ and $\bfp \in L+\bfd$,
it follows that $\bfp - \bfd \in \mathcal{U}\setminus\mathcal{L}$.
Again, there are two subcases:
either
$\bfp -\bfd$ lies in the top edge of $\mathcal{R}$ and $\bfq - \bfd$ is immediately above it,
or 
$\bfp -\bfd$ lies in the right edge of $\mathcal{R}$ and $\bfq - \bfd$ is immediately to its right.
In terms of coordinates,
either
\begin{equation}\label{EqCaseInTheTopEdge}
\bfg_1 < \bfp_1-\bfd_1 \leq \bfs_1, \quad  \bfp_2-\bfd_2 = \bfs_2,\quad 
\bfq_1 = \bfp_1, \quad
\bfq_2 = \bfp_2+1
\end{equation}
or
\begin{equation}\label{EqCaseInTheRightEdge}
\bfp_1-\bfd_1 = \bfs_1,\quad  \bfg_2<\bfp_2-\bfd_2 \leq \bfs_2,\quad
\bfq_1 = \bfp_1+1, \quad \bfq_2 = \bfp_2.
\end{equation}
Again the two cases are symmetric, and we may assume that \Cref{EqCaseInTheTopEdge} holds.

Since $\bfq \in \mathcal{R} \setminus H \subseteq \mathcal{R} \setminus \mathcal{U}$, we have $\bfq_2 < \bfs_2$.
Using \Cref{EqCaseInTheTopEdge}, we conclude that $\bfd_2= \bfq_2 - 1 - \bfs_2 < -1$.

Let $\bfd' = \bfd+(0,1)$. 
Since 
$
\bfp \in \mathcal{C}\subseteq (L\cap \mathcal{R})+\bfd,
$ 
we have $\bfp-\bfd \in L\cap\mathcal{R}$ and, therefore,
$\bfq=\bfp+(0,1) =  (\bfp-\bfd)+\bfd' \in (L\cap\mathcal{R})+\bfd'$.
We also have $\bfq \in \mathcal{R}$ by the assumption of this case, 
and $\bfq \notin H$ because $\bfq \in (L+\bfd) \setminus H$.
In conclusion, we have
$\bfq\in 
\big((L\cap \mathcal{R})+\bfd'\big)\cap \mathcal{R}\setminus H \ne \emptyset$,
contradicting  the maximality of $\bfd$.
This concludes the proof of this case, and of the lemma.
\end{proof}

In conclusion, we have proved the following result,  which establishes \Cref{main_conjecture} for monomial ideals.

\begin{thm} \label{ThmSmoothMonomialClassification} 
Let $[S/I] \in \Hilb^d(\AA^3)$ be such that $I \subseteq S$ is a monomial ideal.
The following conditions are equivalent
\begin{enumerate}
\item\label{it:smoothMonomial1} The point $[S/I]$ is a smooth point of the Hilbert scheme.
\item\label{it:smoothMonomial2} The ideal $I$ admits no singularizing triple.
\item\label{it:smoothMonomial3} The algebra $S/I$ admits a broken Gorenstein structure without flips.
\item\label{it:smoothMonomial4} The algebra $S/I$ admits a broken Gorenstein structure.
\item\label{it:smoothMonomial5} The ideal $I$ is licci.
\end{enumerate}
\end{thm}
\begin{proof}
    The implication $\eqref{it:smoothMonomial1} \implies \eqref{it:smoothMonomial2}$ follows from
    \Cref{ThmSingularMonomial}. The implications $\eqref{it:smoothMonomial2}\implies \eqref{it:smoothMonomial3}$ follows
    from \Cref{PropSmoothMonomial}. The implication $\eqref{it:smoothMonomial3}\implies \eqref{it:smoothMonomial4}$ is
    formal. The implication $\eqref{it:smoothMonomial3}\implies \eqref{it:smoothMonomial5}$ follows from
    \Cref{ref:NoFlipsLicci:thm}. The implication $\eqref{it:smoothMonomial5}\implies \eqref{it:smoothMonomial1}$ is well-known,
    while $\eqref{it:smoothMonomial4}\implies \eqref{it:smoothMonomial1}$ is \Cref{ref:canitbetrue:thm}.
\end{proof}

\begin{remark}[Relation to~\cite{Huibregtse_monomial}]\label{RemHuibregtse}
After completing the first version of this paper, 
we became aware of the preprint
\cite{Huibregtse_monomial} by Mark Huibregtse, 
where the author studies the tangent space to monomial points of $\Hilb^d(\AA^n)$,
with special emphasis on the case $n=3$. 
The main result,
\cite[Theorem 10.3.1]{Huibregtse_monomial},
 characterizes smooth monomial points $[S/I]\in\Hilb^d(\AA^n)$ as those 
whose corresponding staircase $E_I$ is a ``compound box''.
Being a
compound box turns out to be equivalent to the condition
\Cref{PropNoSingularizingTriple}(iii),
thus, \cite[Theorem 10.3.1]{Huibregtse_monomial}  is equivalent to the directions $\eqref{it:smoothMonomial1}
\Leftrightarrow \eqref{it:smoothMonomial2}$ in  \Cref{ThmSmoothMonomialClassification} and the two
classifications agree.
Singularizing triples or similar structures do not appear
explicitly in~\cite{Huibregtse_monomial}.

The approach used in  \cite{Huibregtse_monomial} is
purely combinatorial,  focusing on tangent spaces and relying on the visualization of  tangent vectors as \emph{Haiman arrows}
\cite{Haiman_Catalan}. 
Our approach is different,  incorporating linkage and broken Gorenstein structures  as well as Serre duality \cite{RS22}
into the combinatorics.  
 Moreover,  we prove a stronger result than non-smoothness,
resolving  \cite[Conjecture 4.25]{H23} in
\Cref{ThmSingularMonomial};
as far as we know, this stronger version does not
follow using the method of \cite{Huibregtse_monomial},
since \cite[Case 1,
Lemma~10.1.1]{Huibregtse_monomial} is not symmetric with respect to the $3$
variables.
\end{remark}

\section{Grassmann singularities} \label{sec:mon2}

The goal of this section is to investigate the nature of the singularities in
$\Hilb^d(\mathbb{A}^3)$ of
points of the smoothable component that have tangent space dimension $3d+6$.
Our analysis relies on two main tools.
In \Cref{SubsectionLinkageSingularities},
which works for an arbitrary smooth ambient scheme $X$,
we illustrate how linkage affects singularities of  $\Hilb(X)$.
As by-product, unrelated to the main purposes of this work,
we obtain new information on linkage classes in codimension three.
In \Cref{SubsectionSingularMonomial},
we perform a detailed analysis of the structure of monomial ideals in
$\kk[x,y,z]$ from the perspectives of linkage and of the combinatorial
framework  employed in \Cref{sec:mon1}, with the aim of proving Hu's
conjectures. One of the punchlines of this part is that the singularities with tangent space
dimension $3d+6$ are, in many cases, smoothly equivalent to the cone over the Pl\"ucker
embedding of $\Gr(2, 6)$, the Grassmannian of two-planes in $\kk^{\oplus 6}$.

\subsection{Linkage and singularities}\label{SubsectionLinkageSingularities}
Let $X$ be a smooth $\kk$-scheme and let $\nestedHilbX{d}{d'}$ denote the nested Hilbert
scheme of points. 
A $\kk$-point of this scheme corresponds to a pair of subschemes
$Z\subseteq Z'\subseteq X$, where $Z$ and $Z'$ are finite of degree $d$ and $d'$,
respectively. 
We denote such a closed point by $[Z\subseteq Z']$.

\begin{definitions}
We define the \textbf{locus of tuples of points} as the open subscheme of
$\nestedHilbX{d}{d'}$ consisting of pairs $[Z\subseteq Z']$ such that $Z'$ is
smooth. 
The name is justified, because over an algebraically closed field
$\kk$, a $\kk$-point of this locus is a pair $Z\subseteq Z'$, where $Z$, $Z'$
are tuples of reduced points.
Let
$$
\nestedlcilocus{d}{d'}\subseteq \nestedHilbX{d}{d'}
$$
be the locus that consists of $[Z \subseteq Z']$ with $Z'$ a locally
complete intersection.
This locus is open and inherits a natural scheme structure, as it is the preimage of 
$\Hilb^{d'}_{\text{lci}}(X)$ under the natural projection 
$$
\nestedHilbX{d}{d'} \to \Hilb^{d'}(X)$$
 and $\Hilb^{d'}_{\text{lci}}(X)$ is open in 
$\Hilb^{d'}(X)$, see \cite[\href{https://stacks.math.columbia.edu/tag/06CJ}{Tag 06CJ}]{stacks_project}.

\end{definitions}

\begin{prop} \label{smooth_equivalence_lci_superscheme}
The projection map $p\colon \nestedlcilocus{d}{d'}\to \Hilb^d(X)$ is smooth. 
The preimage of the smoothable component of $\Hilb^d(X)$ is equal (as a closed subset) to the closure of the locus of tuples of points in $\nestedlcilocus{d}{d'}$.
\end{prop}

\begin{proof}
\def\tilC{\widetilde{C}}%
To prove that $p$ is smooth, we verify the infinitesimal lifting criterion in
its Artinian version
\cite[Proposition~1.1]{Artin_theorems_of_representability}.
Let $\Spec(A_0)\subseteq \Spec(A)$ be a closed immersion of finite local $\kk$-schemes such that  $I = \ker(A\to A_0)$ satisfies $I^2 =0$. 
Consider a commutative diagram
    \[\begin{tikzcd}
	{\nestedlcilocus{d}{d'}} & {\Spec(A_0)} \\
	{\Hilb^d(X)} & {\Spec(A)}
	\arrow[from=1-1, to=2-1]
	\arrow[from=1-2, to=1-1]
	\arrow[from=1-2, to=2-2]
	\arrow[from=2-2, to=2-1]
\end{tikzcd}\]
and let $\mcZ_0\subseteq \mcZ_0' \subseteq X \times \Spec(A_0)$ and $\mcZ\subseteq X \times \Spec(A)$ be the families corresponding to the horizontal maps. 
We need to show that there exists a finite flat family $\mcZ'\subseteq X \times \Spec(A)$ containing $\mcZ$ and restricting to  $\mcZ'_0$.
    
Since $A$ is finite and local, it is a complete local ring. By \cite[Corollary 7.6]{EisView}, the families $\mcZ$ and $\mcZ'_0$ are a disjoint union of families each of which is, topologically, a point.
We can search for $\mcZ'$ point-by-point, 
so we restrict to a point and in particular we have that $\mcZ'_0$ is defined by a regular sequence. 
Let $\mcZ'$ be given by any lift of this sequence to $\msI_{\mcZ}$.
Since $I$ is nilpotent, the inclusion $\mcZ'_0\subseteq \mcZ'$ is an isomorphism on underlying topological spaces. 
In particular the schemes $\mcZ'_0$, $\mcZ'$ have the same dimension, so the lift is again a regular sequence. 
By the syzygetic criterion for flatness, the scheme  $\mcZ'$ is flat over $\Spec(A)$. 
It is finite as well, since $\mcZ'_0$ is finite and $I$ is nilpotent \cite[Tag~00DV, nilpotent Nakayama]{stacks_project}.

By definition, the closure of the locus of tuples of points is contained in the preimage of
the smoothable component. 
To prove the other inclusion, consider a smoothable subscheme $Z_0 \subseteq X$ and a point $[Z_0\subseteq Z_0']\in \nestedlcilocus{d}{d'}$.
Let $\mcZ$ be a family of degree $d$ subschemes over $\Spec(\powerseries)$ with a smooth generic fiber and passing through $Z_0$. 
By the smoothness of $p$, this lifts to a family $[\mcZ \subseteq \mcZ']$ in $\nestedlcilocus{d}{d'}$ passing through $[Z_0\subseteq Z'_0]$. 
Thus, it is enough to prove that a general point of this curve lies in the closure of the locus of tuples of points. 
In particular, any such point is of the form $[Z\subseteq Z']$ where $Z$ is a reduced union of points and $Z'$ is a locally complete intersection. 
Fix such a pair $Z \subseteq Z'$ for the rest of the proof.
    
By~\cite[Theorem~3.10]{huneke_ulrich_cis}, the subscheme $Z'\subseteq X$ is smoothable. Consider a smoothing  $\mcZ'\subseteq X\times C$, where $(C,0)$ is an irreducible  curve and $\mcZ'|_0 = Z'$.
It is well-known (see, for example, \cite[Proposition~2.6]{jelisiejew_keneshlou}),
there is a finite surjective base change map $\tilC\to C$, where $\tilC$ is irreducible,
 and sections $s_1, \ldots , s_{d'}\colon \tilC \to \mcZ'\times_C \tilC$, such that
    \[
        \mcZ'\times_{C}\tilC = \bigcup_{i=1}^{d'} s_i(\tilC).
    \]
Pick sections $s_{i_1}$, \ldots ,$s_{i_d}$ such that $\left(s_{i_1}(\tilC) \cup  \cdots \cup s_{i_d}(\tilC)\right)|_{0} = Z$. 
Up to  shrinking $\tilC$, we may assume that no  two of these  sections intersect.
In particular, $\widetilde{\mcZ} := s_{i_1}(\tilC) \cup  \cdots \cup s_{i_d}(\tilC)$ is finite flat over $\tilC$ of degree $d$. 
This
yields an irreducible curve $\tilC \hookrightarrow \Hilb^{(d,d')}(X)$  whose general point  lies in locus of tuples of points. 
Since this curve also passes through the point $[Z \subseteq Z']$, 
this point must lie in the  closure of the locus of tuples of points.
\end{proof}

\begin{prop}[Linkage in families]\label{linkage_in_families}
 For any non-negative integers $d\leq d'$, there is an isomorphism
    \[
        L_{d,d'}\colon \nestedlcilocus{d}{d'}\to \nestedlcilocus{{d'-d}}{d'}
    \]
    which sends a family $\mcZ\subseteq \mcZ'$ to $\mcZ''\subseteq \mcZ'$,
    where $\mcZ'' := V(\Ann(\msI_{\mcZ\subseteq
\mcZ'}^{\vee}))$. 
Moreover, the composition $L_{d'-d, d'} \circ L_{d, d'}$ is the  identity map.
\end{prop}

\begin{proof}
    Fix any base scheme $B$ and a family $\mcZ \subseteq \mcZ'$ corresponding
    to a $B$-point of $\nestedlcilocus{d}{d'}$. The map $\mcZ'\to B$ is
    affine; we will identify sheaves on $\mcZ'$ with sheaves of
    $\mcO_{\mcZ'}$-algebras on $B$.
The structure sheaves $\msO_{\mcZ}$ and $\msO_{\mcZ'}$ are locally free $\msO_B$-modules of rank $d$ and $d'$, respectively.
The ideal sheaf $\msI_{\mcZ \subseteq \mcZ'}$ is the kernel of the surjection $\msO_{\mcZ'}\onto \msO_{\mcZ}$, so it is also a locally free $\msO_B$-module of rank $d'-d$. This implies that the sheaf
    \[
        \msI_{\mcZ\subseteq \mcZ'}^{\vee} = \Hom_{\mcO_B}\left( \msI_{\mcZ\subseteq
        \mcZ'}, \mcO_B \right)
    \]
    commutes with base changes $B' \to B$.
    Consider $\omega_{\mcZ'} = \Hom_{\mcO_B}(\mcO_{\mcZ'}, \mcO_B)$, $\omega_{\mcZ} =
    \Hom_{\mcO_B}(\mcO_{\mcZ}, \mcO_B)$ with their usual $\mcO_{\mcZ'}$ and
    $\mcO_{\mcZ}$-module structures. Since $\mcZ'\to B$ has Gorenstein fibers,
    the $\mcZ'$-module $\omega_{\mcZ'}$ is invertible.
    We have an exact sequence
    \begin{equation}\label{eq:surj}
        0 \to \omega_{\mcZ} \to \omega_{\mcZ'} \to \msI_{\mcZ\subseteq
        \mcZ'}^{\vee} \to 0.
    \end{equation}
Let $\mcZ'' = V(\Ann_{\mcO_{B\times X}}(\msI_{\mcZ\subseteq
\mcZ'}^{\vee}))\subseteq B \times X$. Locally on $B$, the  $\mcZ'$-module
$\omega_{\mcZ}$ trivializes. On each open $U\subseteq B$ trivializing it, the sequence~\eqref{eq:surj} becomes
\begin{equation}\label{eq:surjDual}
    0 \to \omega_{\mcZ}|_{U}\to \mcO_{\mcZ'} \to \mcO_{\mcZ''} \to 0,
\end{equation}
which shows that locally on $B$, the $\mcO_B$-modules $\mcO_{\mcZ''}$ and $\msI_{\mcZ\subseteq
\mcZ'}^{\vee}$ are isomorphic. This implies that $\mcZ''\to B$ is finite flat of
degree $d'-d$, so the map $L_{d, d'}$ is well-defined.

To show that $L_{d'-d, d'} \circ L_{d, d'}$ is the identity, we can work
locally on $B$, so we restrict to $U$. Performing the above construction starting
from~\eqref{eq:surjDual}, we obtain a closed subscheme $\mcZ'''\subseteq
U\times X$ which over $U$ is given by the annihilator of
$\omega_{\mcZ}|_U^{\vee}$. We see that $\mcZ''' = \mcZ|_U$.
Thus the composition $L_{d'-d, d'} \circ L_{d, d'}$ is the identity for any
$d$. In particular, $L_{d,d'}$ is an isomorphism.
\end{proof}

The result above has two important consequences regarding how the singularities of the Hilbert scheme change under linkage. The first is that linkage preserves the smoothable tangent excess, as defined in \Cref{def_tangent_excess}.

The following result is folklore \cite{BuchweitzThesis,BuchweitzUlrich}, we include the proof for completeness.

\begin{thm} \label{theorem_excess_dim} 
Let $X$ be a smooth irreducible $n$-dimensional $\kk$-scheme.
Let $[Z] \in \Hilb^d(X)$ and $[Z''] \in \Hilb^{d''}(X)$ and assume that $Z,Z''$ are smoothable. 
If $Z$ is linked to $Z''$, then 
$$
\dim_{\kk} T_{[Z]} \Hilb^d(X) - d\cdot n  = \dim_{\kk} T_{[Z'']}
\Hilb^{d''}(X) - d'' \cdot n.
$$
\end{thm}

\begin{proof}
    Assume $Z$ is linked to $Z''$ by a complete intersection $Z' := V(\ua)$.
    Consider the point $[Z \subseteq Z']\in \nestedlcilocus{d}{d'}$ and the projection map $p\colon
    \nestedlcilocus{d}{d'}\to \Hilb^d(X)$.
    By \Cref{smooth_equivalence_lci_superscheme}, the smoothable tangent excess at $Z$
    is equal to the difference $\delta_{Z\subseteq Z'}$ between the dimension of the tangent space at $[Z \subseteq Z']$ and the
    dimension of the locus of tuples of points in $\nestedlcilocus{d}{d'}$.

    \Cref{linkage_in_families} yields an isomorphism of schemes
    $\nestedlcilocus{d}{d'}$ and $\nestedlcilocus{d''}{d'}$ which maps
    $[Z\subseteq Z']$ to $[Z''\subseteq Z']$. This isomorphism, by
    definition, is an isomorphism on the locus of tuples of points. It follows that $\delta_{[Z\subseteq
    Z']}$ and $\delta_{[Z''\subseteq Z']}$ are equal. 
\end{proof}

Similarly,
another consequence is the fact that the singularities at points on the Hilbert scheme, 
whose corresponding ideals are linked, 
are smoothly equivalent.

\begin{definition}[\cite{Vakil_MurphyLaw, Jelisiejew__Pathologies}]\label{ref:smoothlyEquivalent:def}
Two pointed schemes $(X, x)$, $(Y, y)$ are \textbf{smoothly equivalent} if there
exists a third pointed scheme $(Z, z)$ with {smooth} maps $(Z, z)\to (X,
x)$, $(Z, z)\to (Y, y)$.
\end{definition}

 Intuitively, smoothly equivalent points have the same 
  geometry up to  free parameters.

\begin{thm} \label{theorem_smoth_equiv} 
Let $[Z] \in \Hilb^d(X)$ and $[Z''] \in \Hilb^{d''}(X)$. 
If $Z$ is linked to $Z''$,
then the singularity at $[Z]$ is smoothly equivalent to the singularity at $[Z'']$. 
\end{thm} 
\begin{proof} This follows immediately from \Cref{smooth_equivalence_lci_superscheme} and \Cref{linkage_in_families}.
\end{proof}

In \cite[p. 389]{Huneke_Ulrich__Monomials}, the authors ask when a zero-dimensional ideal $I \subseteq S = \kk[x_1, \ldots, x_n]$ belongs to the linkage class of a monomial ideal.
Since the licci class is the only linkage class for $n=2$, the first interesting case of this question occurs when $n=3$.
As a byproduct of our work, we obtain a method to explicitly produce 
many ideals that do not even belong to the linkage class of a homogeneous ideal.

\begin{cor} \label{corollary_linkage_obstruction} 
Let $I \subseteq S = \kk[x,y,z]$  be an ideal with $\dim_\kk(S/I) =d$.
\begin{enumerate}
    \item If $S/I$ is not smoothable, then $I$ is not in the linkage class of
        any monomial ideal.
    \item If $\dim_{\kk} T(I) \not \equiv d \bmod 2$, then $I$ is not in the linkage
        class of any ideal homogeneous with respect to the standard grading.
\end{enumerate}
\end{cor}

\begin{proof}
It is known that smoothability is preserved under linkage.
This also follows from \Cref{linkage_in_families}, 
since \Cref{smooth_equivalence_lci_superscheme} shows that the smoothable component is
exactly the image of closure of the locus of tuples of points.
Since  monomial ideals are smoothable \cite[Proposition~4.10]{CEVV},
 it follows that a nonsmoothable algebra $S/I$
cannot be in the linkage class of a monomial ideal.

By \cite[Theorem 1]{Ramkumar_Sammartano_parity},  
homogeneous ideals $J \subseteq S$  
satisfy $\dim_{\kk} T(J)  \equiv \dim_{\kk} (S/J) \bmod 2$.
It follows by \Cref{theorem_excess_dim} that,
if $\dim_{\kk} T(I) \not \equiv d \bmod 2$, then $I$ cannot be in the linkage class
of a homogeneous ideal.
\end{proof}

Thanks to the work~\cite{Graffeo_Giovenzana_Lella}, ideals with odd smoothable tangent excess are known.

\begin{example} \label{example_GGGL} Consider the binomial ideal
\[
I = (x+(y,z)^2)^2 + (y^3-xz) = (x^2,xy^2,xyz,,xz^2,y^2z^2,yz^3,z^4,y^3-xz).
\]
It is shown in \cite{Graffeo_Giovenzana_Lella} that $\dim_\kk(S/I) = 12$, 
while $\dim_{\kk} T(I) = 45$. 
It follows from \Cref{corollary_linkage_obstruction} that $I$ is not in the linkage class of  a monomial ideal. 
The ideal $I$ arises from the monomial
ideal $J = (x+(y,z)^2)^2$ by adding the binomial $y^3-xz$, which lies in the socle of $S/J$. 
In fact, dividing $S/J$ by a \emph{general} socle element
yields a quotient $S/I'$ with $\dim_{\kk} T(I') = 45$, 
see~\cite[\S3]{Graffeo_Giovenzana_Lella}.
As explained by Giovenzana-Giovenzana-Graffeo-Lella (private communication), 
similar constructions yield many more examples of monomial ideals with odd smoothable tangent excess,  see~\cite{Kool_Jelisiejew_Schmierman} for another example.
\end{example}

\subsection{Singular monomial ideals}\label{SubsectionSingularMonomial}
Our next goal is to investigate monomial ideals that give rise to the mildest possible singularities on the Hilbert scheme, 
in the sense of \Cref{ThmSingularMonomial}.
In particular, in this subsection we will study  their linkage classes and tangent spaces.

One of our goals is to understand monomial ideals $I\subseteq S = \kk[x, y, z]$ with given tangent dimension,
such as $3d+6$. For this, we need to construct tangent vectors. 
    A homomorphism $\varphi\colon I\to S/I$ is a \textbf{socle map} if its
    image is contained in $\soc(S/I)$.
A socle map yields a map $I/\mm I \to S/I$, where $\mm := (x, y, z)$, 
and,
conversely, any $\kk$-linear map $I/\mm I\to \soc(S/I)$ yields a socle map.

\begin{example} \label{example_tripod_1} A \textbf{tripod} is an ideal of the form $I^{\text{tri}}(a,b,c) := (x^a,y^b,z^c,xy,xz,yz)$ for some $a,b,c\geq 2$.    
The associated staircase~\cite[p. 46]{MS05} explains the choice of terminology: 
$$ 
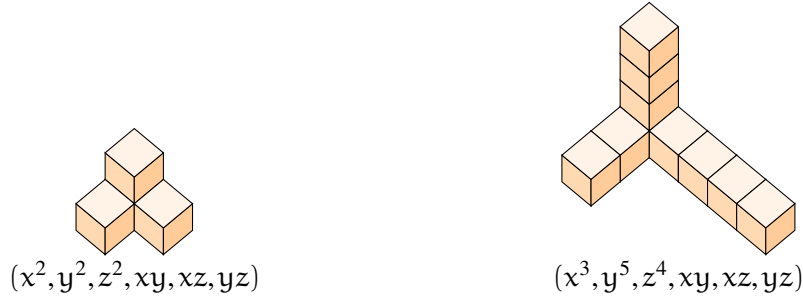
\underset{\displaystyle(x^2,y^2,z^2,xy,xz,yz)}{\monomialIdealPicture{{2,1},{1,0}}}
\qquad \qquad \qquad \qquad \qquad
\underset{\displaystyle{(x^3,y^5,z^4,xy,xz,yz)}}{\monomialIdealPicture{{4,1,1},{1,0,0},{1,0,0},{1,0,0},{1,0,0}}}
$$
Notice that $\soc(S/I^{\tri}(a,b,c))$ is spanned by
$\{x^{a-1},y^{b-1},z^{c-1}\}$, and this triple is a singularizing one.
Geometrically, these correspond to the three corners, with each corner is
maximal in one of the directions.

Since $I^{\tri}(a,b,c)$ has six minimal generators,
 there are $3\cdot 6 = 18$ linearly independent socle maps.
Among the socle maps, the doubly-negative (see \Cref{ref:doublyNegative:def})
ones are of the form $\varphi(xy) = z^{a-1}$, $\varphi(xz) = y^{b-1}$ and
$\varphi(yz) = x^{a-1}$.
 It is possible to show that  these are the only doubly-negative tangents of
 $I^{\tri}(a,b,c)$.
Thus, by \Cref{theorem_smooth_monomials} we have $\dim_{\kk} T(I^{\tri}(a,b,c)) =  3d+6$. 
\end{example}

Among monomial ideals, an important special class is that of strongly stable ideals. 
A monomial ideal $I\subseteq \kk[x_1,\dots,x_n]$ is said to be \textbf{strongly stable} if  for every minimal monomial generator $m\in I$ and for every $x_j$ dividing $m$, we have $\frac{x_i}{x_j}m \in I$ for all $i < j$.
Strongly stable ideals have rich combinatorial structure, and they are  Borel-fixed.
If the field $\kk$ has characteristic zero, 
being strongly stable and being Borel-fixed are equivalent conditions
\cite[Proposition 2.3]{MS05}.

We now describe the socle monomials of a cofinite strongly stable ideal.

\begin{prop}\label{ConstructSocle} Let $I \subseteq S$ be a cofinite strongly stable ideal.
Let $m' \in S/I$ be a non-zero monomial and let $\gamma$ be maximal so that $m := z^\gamma m' \in S/I$ is non-zero.
Then, $m$ is a socle monomial.
Moreover, all socle monomials are of this form.
\end{prop}
\begin{proof} Given $m$ as in the statement, the maximality of $\gamma$ implies that $zm \in I$. Since $I$ is strongly stable,  we have that $xm,ym \in I$ and, in particular, $m\in \soc(S/I$).
Working backwards, we  see that all socle monomials arise in this way.
\end{proof}

We move on to construct tangent vectors of a strongly stable  ideal $I
\subseteq S = \kk[x, y, z]$. We observe that thanks to stability, the ideal $I$
has some particular generators. For example, there is always a generator of
the form $xy^b$ for some $b\geq 0$. Indeed, since $S/I$ is
finite-dimensional, there is a minimal generator of the form $y^{e}$ with $e\geq 1$.
By stability, there is an element of $I$ of the form $xy^{e-1}$. This implies
that there is a generator of this form as well.

We are  ready to provide the ``only if'' part of the classification of
strongly stable ideals in $\Hilb^d(\mathbb{A}^3)$ with tangent space dimension $3d+6$.
Let $a,b,c \in \N$ be integers such that  $1 \leq a \leq b \leq c$.
Define the ideal
\[
J(a,b,c) := (x^2,xy,y^2,xz^a,yz^b,z^{c+1}).
\]
Observe that it is a strongly stable ideal of codegree $a+b+c+1$.
\begin{prop}\label{ref:stronglyStableClassification:prop}
Let $I\subseteq S$ be a strongly stable ideal such that $\dim_\kk(S/I)=d$.
If $\dim_{\kk} T(I) = 3d+6$, then $I = J(a,b,c) = (x^2,xy, y^2,xz^a,yz^b,z^{c+1})$ with $a\leq b \leq c$ and either  $a = 1$ or $b=c$.
\end{prop}
\begin{proof}
We may assume $x \notin I$, because otherwise $[I]$ would define a point of $\Hilb^d(\mathbb{A}^2)$ and would therefore be smooth. We begin by constructing some doubly-negative tangent vectors.

An \textbf{nnp}-tangent.
Since $x \notin I$, there is a minimal generator of the form $xy^b$ with $b > 0$.
By \Cref{ConstructSocle}, there is a (unique) socle monomial of the form $m_{\soc} = z^\gamma$.
Then, the socle map defined by $\varphi(xy^b) = m_{\soc}$  lies in $T_{\nnp}(I)$ with weight $(-1,-b,\gamma)$.

A \textbf{pnn}-tangent.
Since $x \notin I$, we have $y \notin I$, and so there is a minimal generator $yz^c$ with $c > 0$.
By \Cref{ConstructSocle}, there is a socle monomial $m_{\soc} = x z^\gamma$
with $\gamma \geq 0$. Since $I$ is strongly stable, $xz^c \in I$ and thus we
must have $\gamma < c$. It follows that the socle map defined by $\varphi(yz^c) = m_\soc$ lies in $T_{\pnn}(I)$ with weight $(1,-1,\gamma-c)$.

An \textbf{npn}-tangent.
Unlike the above two cases, the npn-map we will construct needs not be a socle map.
Since $x \notin I$, there is a minimal generator of the form $xz^c$ with $c > 0$.
Let $b\geq 2$ be such that $y^b$ is a minimal generator of $I$.
Then, we claim that the map defined by
$\varphi(xz^c) =  y^{b-1}z^{c-1}$ and $\varphi(m) = 0$ for all other minimal
generators  of $I$ yields a tangent vector $\varphi$ in $T_{\npn}(I)$ with
weight $(-1,b-1,-1)$.

To prove the claim, we need to show that if $p_1m_2-p_2m_1 = 0$ is a syzygy,  
with $p_i \in S$ and $m_i \in I$, then $p_1\varphi(m_2) = p_2\varphi(m_1)$.  
Since $x$ and $y$ annihilate $y^{b-1}z^{c-1}$ in $S/I$, we only need to check the minimal syzygies for which $p_1$ or $p_2$ is a pure power of $z$. We may assume $p_1 = z^j$ and $m_2 = xz^c$. This forces, $p_2 = x$ and $m_1 = z^\gamma \in I$ is a minimal generator, and $\gamma = j+c$. Since, $xz^k\varphi(z^{\gamma}) =0$ we only need to show that $z^j\varphi(xz^c) =  0$. But this follows from the fact that 
$z^j\varphi(xz^c) = z^j(y^{b-1}z^{c-1}) = y^{b-1}z^{j+c-1} = y(y^{b-1}z^{\gamma-1}) \in I$.
This concludes the proof of claim.

Suppose that $\dim_{\kk} T(I) = 3d+6$.
By \Cref{theorem_smooth_monomials} it follows that 
 $\dim_{\kk} T_{\pnn}(I) = \dim_{\kk} T_{\npn}(I) =\dim_{\kk} T_{\pnn}(I) =
 1$. If we produce a doubly-negative tangent vector different from the above
 ones, we get a contradiction.

Assume that $xy \notin I$. 
Then,
in addition to  the doubly-negative tangent vectors constructed above, 
we find another doubly-negative tangent vector in $T_{\pnn}(I)$. 
Indeed, the strongly stable property implies that $y^2 \notin I$ and thus there is a minimal generator of the form $m_\gen = y^2z^c$. 
By \Cref{ConstructSocle}, there is a socle monomial $m_\soc = xyz^\gamma$. 
Again, by the strongly stable property, we have $\gamma < c$.
 It follows that the socle map defined by $\varphi(m_\gen) = m_\soc$ lies in $T_{\pnn}(I)$ with weight $(1,-1,\gamma-c)$.

We thus assume that $xy\in I$.
Assume now that $y^2 \notin I$. 
Let $z^\gamma \in I$ be a minimal generator, and since $I$ is strongly stable, we have $yz^{\gamma-1}, y^2z^{\gamma-2}, xz^{\gamma -1} \in I$. 
Since $x,y, y^2 \notin I$, this implies that we have minimal generators of the form $xz^i, y^2z^j, yz^k \in I$ with $i \leq k$, $j \leq k-1$ and $k \leq \gamma-1$. 
It follow that $xz^{i-1}, y^2z^{j-1}, yz^{k-1}$ and $z^{\gamma-1}$ are in $\soc(S/I)$. Note that 
\begin{itemize}
\item if $i \leq j$, there is a socle map induced by $\varphi(y^2z^j) = xz^{i-1}$ with weight $(1,-2,i-1-j)$, and
\item if $i > j$, there is a socle map induced by $\varphi(xz^i) = y^2z^{j-1}$ has weight $(-1,2,j-1-i)$. 
\end{itemize}
In both cases, we get doubly negative maps different from those constructed
above  (compare the positions of the weights that are less than $-1$), a contradiction. Thus $y^2 \in I$, and $I= (x^2,xy,y^2,xz^{a},yz^b,z^{c+1})$ for some $1 \leq a \leq b \leq c$.

Finally, assume $1 < a \leq b < c$.  
Then, there is a tangent vector $\varphi$ in $T_{\nnp}(I)$ with weight $(-1,-1,c-1)$  such that $\varphi(xy) = z^{c-1}$ and $\varphi(m)= 0$ for all other minimal generators $m$ of $I$. 
To see that this is well-defined, note that $z^{c-1}$ is annihilated by
$x,y,z^2$ in $S/I$. Then, arguing as above, the only minimal syzygies we need to check are of the form $p_1m_1-p_2m_2$ with some $p_i$ equal to either, $1$ or $z$. 
However, looking at the minimal generators of $I$ we see that there is no such syzygy. 
Thus, $\varphi$ is well-defined. 
Since the weight of $\varphi$ is different from the weight of the
$\nnp$-vector constructed above, we obtain $\dim_{\kk} T(I) > 3d+6$.
\end{proof}

Combining this analysis with the linkage technique from
\Cref{SubsectionLinkageSingularities},
we are able to classify all the strongly stable ideals with tangent space dimension $3d+6$.

\begin{prop} \label{prop_linkage_Borel_3d+6}
The following ideals of $S$ belong to the linkage class of $\mm^2$:
\begin{itemize}
\item $J(a,b,c)$, if $a = 1$ or $b =c$;
\item $I^{\tri}(a,b,c)$, for all $a,b,c \geq 2$.
\end{itemize}
\end{prop}
\begin{proof} 
Since $\mm^2 = I^{\tri}(2,2,2)$, it suffices to establish the following links
\begin{enumerate}
\item \label{BorelItem1} $I^{\tri}(a,b,c) \sim I^{\tri}(2,2,c)$,
\item \label{BorelItem2} $J(1,b,c) \sim I^{\tri}(2,2,c-b+2)$, and
\item \label{BorelItem3} $J(a,b,b) \sim J(1,b-a+1,b)$.
\end{enumerate}
Indeed, by symmetry, \eqref{BorelItem1} also implies $I^{\tri}(a,b,c) \sim I^{\tri}(a,2,2)$,
and therefore $I^{\tri}(2,2,c) \sim I^{\tri}(2,2,2)$.

We begin by proving \eqref{BorelItem1}.
Consider the sequence
$$\ua := \big(xy,xz+yz,x^a+y^b+z^c\big)\subseteq I^{\tri}(a,b,c).$$
Choosing the pure lexicographic order with $z>x>y$,  the initial ideal of $(\ua)$ is $(xy, xz, zy^2, x^{a+1}, y^{b+2}, z^c)$.
It follows that a $\kk$-basis of $S/(\ua)$ is 
$$
\big\{1,x,\dots,x^{a},y,\dots,y^{b+1},yz,\dots,yz^{c-1},z,\dots,z^{c-1}\big\},
$$
thus,  $\ua$ generates an ideal of codimension 3,  it is a  regular sequence, and  $\dim_\kk(S/(\ua)) = a+b+2c$. 
Next, we observe that  that $I^{\tri}(a,b,c)I^{\tri}(2,2,c)
\subseteq (\ua)$, that is,
$I^{\tri}(2,2,c) \subseteq (\ua:I^{\tri}(a,b,c))$.
We have 
$$
xy \equiv 0 \bmod (\ua), 
\qquad
x^2z \equiv -xyz \equiv 0 \bmod (\ua), 
\qquad
y^2z \equiv - xyz \equiv 0 \bmod (\ua),
\qquad
$$
$$
xz^{c+1}  \equiv -xz(x^a+y^b) \equiv -x^{a+1}z \equiv x^{a}yz \equiv 0 \bmod (\ua),
$$
$$
yz^{c+1} \equiv -yz(x^a+y^b) \equiv -y^{b+1}z \equiv y^{b}xz \equiv 0 \bmod (\ua).
$$
Thus, all the mixed monomials in the product  $I^{\tri}(a,b,c)I^{\tri}(2,2,c)$ also lie in $(\ua)$. 
It remains to check $x^{a+2}$, $y^{b+2}$ and $z^{2c}$ lie in $(\ua)$. It is enough to show that these are equivalent to mixed monomials modulo $(\ua)$. 
Indeed, we have $x^{a+2} \equiv -x^2(y^b+z^c) \bmod (\ua)$, $y^{b+2} \equiv -y^2(x^b+z^c) \bmod (\ua)$ and $z^{2c} \equiv - z^{2c}(x^a+y^b) \bmod (\ua)$. 
Finally, since 
$\dim_\kk( S/I^{\tri}(2,2,c)) + \dim_\kk( S/I^{\tri}(a,b,c)) = \dim_\kk( S/\ua)$,
it follows that $I^{\tri}(2,2,c) = (\ua:I^{\tri}(a,b,c))$,
and this shows the desired claim \eqref{BorelItem1}.

We apply the same argument to the remaining items.
For  \eqref{BorelItem2}, we use the regular sequence  
$$\ua = \big(xz,y^2,z^{c+1}+x^2\big) \subseteq J(1,b,c).$$
The initial ideal of $(\ua)$ is
 $(xz, y^2,x^3,z^{c+1})$,
 thus,
$$\{1,x,x^2,xy,x^2y,y,yz,\dots,yz^c,z,\dots,z^{c}\}$$
is  a $\kk$-basis of $S/(\ua)$, and
 $\dim_\kk( S/\ua)  = 2c+6 =\dim_\kk( S/J(1,b,c)) + \dim_\kk( S/I^{\tri}(2,2,c-b+2))$.
As before,
it remains to verify that  $J(1,b,c)I^{\tri}(2,2,c-b+2)\subseteq (\ua)$. 
We have $xz, y^2x,y^2z \equiv 0 \bmod (\ua)$, $yx^3 \equiv - yxz^{c+1} \equiv 0 \bmod (\ua)$ and $yz^{c+2} \equiv - yzx^2 \equiv 0 \bmod (\ua)$. It follows that all the mixed monomials of $J(1,b,c)I^{\tri}(2,2,c-b+2)$ lie in $(\ua)$. To deal with the pure powers note that $x^{4} = -x^2z^{c+1} \equiv 0 \bmod (\ua)$, $y^2 \equiv 0 \bmod (\ua)$ and $z^{c+2} \equiv -zx^2 \equiv 0 \bmod (\ua)$. In conclusion, $J(1,b,c)\sim I^{\tri}(2,2,c-b+2)$.

For item \eqref{BorelItem3}, we use the regular sequence $(\ua) = (x^2,y^2,xy+z^{b+1}) \subseteq J(a,b,b)$. 
Clearly, $ (x^2,y^2,z^{b+1})$ is an initial ideal of $(\ua)$, 
so $\dim_\kk(S/(\ua)) = 4b+4= \dim_\kk( S/J(a,b,b)) + \dim_\kk( S/J(1,b-a+1,b))$.
Once again, it suffices to show $J(a,b,b)J(1,b-a+1,b)\subseteq (\ua)$. We have $x^2y,xy^2, x^2z,y^2z \equiv 0 \bmod (\ua)$, $xz^{b+1} \equiv - x^2y \equiv 0 \bmod (\ua)$ and $yz^{b+1} \equiv - xy^2 \equiv 0 \bmod (\ua)$.  
Hence, all the mixed monomials of $J(a,b,b)J(1,b-a+1,b)$ lie in $(\ua)$. To deal with the pure powers note that $x^2,y^2 \equiv 0 \bmod (\ua)$ and $z^{2(b+1)} \equiv -(xy)^2 - 2(xy)(z^{b+1}) \equiv 0 \bmod (\ua)$. In conclusion, $J(a,b,b)\sim J(1,b-a+1,b)$.
\end{proof}

We now present  the main result of this subsection, which settles
{\cite[Conjectures 4.27 and 4.29]{H23}}.

\begin{thm}\label{theorem_strongly_stable_3d+6}
Let $I\subseteq S$ be a strongly stable ideal such that $\dim_\kk(S/I)=d$.
Then, $\dim_{\kk} T(I) = 3d+6$ if and only if $I = J(a,b,c) = (x^2,xy, y^2,xz^a,yz^b,z^{c+1})$ with $a\leq b \leq c$ and either  $a = 1$ or $b=c$.
\end{thm}
\begin{proof} 
Suppose  $I = J(a,b,c)$ with $a\leq b \leq c$ and either  $a = 1$ or $b=c$.
Then, by \Cref{prop_linkage_Borel_3d+6} and \Cref{theorem_excess_dim}, it follows that $T(I)$ has the same smoothable tangent excess as $T(\mm^2)$, namely, $6$ \cite[Proposition III.4]{Briancon_Iarrobino__dimension_of_punctual_Hilbert_scheme}, as required. 
The converse is \Cref{ref:stronglyStableClassification:prop}.
\end{proof}

\begin{remark} \label{remark_counterexample_p_Borel} 
When $\kk$ is a field  of characteristic $ p > 0 $,
there exist  Borel-fixed ideals which are not strongly stable ideals. 
We show that  \Cref{theorem_strongly_stable_3d+6} cannot be extended to the class of Borel-fixed ideals.

Let $p$ be an odd prime. The ideal
$
I = (x^2,xy,xz,y^p,(yz)^{p-1},z^p)
$
is Borel-fixed by \cite[Theorem 15.23]{EisView}. 
Consider $\ua := (xy,xz + y^p, x^2 + z^p) \subseteq I$.
One can argue as in the proof of \Cref{prop_linkage_Borel_3d+6} to show that
$
(\ua:I) = (x^2,y^2,z^p,xy,xz,yz) = I^{\tri}(2,2,p).
$
Thus, as noted in the proof of the previous theorem, it follows that $\dim_{\kk} T(I) = 3d+6$. 
However, it is not of the form described in \Cref{theorem_strongly_stable_3d+6}.
\end{remark}

We now proceed to
we determine the singularity type for 
singular points on  the Hilbert scheme
corresponding to tripod ideals and strongly stable ideals with a smoothable tangent excess of 6.
In this way, we affirmatively answer \cite[Conjecture 4.31]{H23} in many cases.

\begin{thm}\label{ref:surplusSixIsGrCone:thm} 
Let $I\subseteq S$ be a strongly stable ideal or a tripod ideal
with $\dim_\kk(S/I)=d$.
Assume that $\dim_{\kk} T(I) = 3d+6$.
Then, the singularity at $[S/I] \in \Hilb^d(\AA^3)$ is smoothly equivalent
to the vertex of a cone over the Grassmannian
$\Gr(2,6) \hookrightarrow \mathbb{P}^{14}$ in its Pl\"ucker embedding.
\end{thm}
\begin{proof}
 By \Cref{prop_linkage_Borel_3d+6}, $I$ is in the linkage class of $\mm^2$. 
 By \cite{Katz__4points}, the singularity at $[\mm^2] \in \Hilb^4(\AA^3)$ is locally a cone over $\Gr(2,6) \hookrightarrow \mathbb{P}^{14}$ in its Pl\"ucker embedding with a $3$-dimensional vertex. 
 The result now follows by applying \Cref{theorem_smoth_equiv}.
\end{proof}

The above proof  suggests that an approach to proving \cite[Conjecture 4.31]{H23} in full generality would involve showing that all ideals $I$ with tangent space dimension $3d+6$ are in the  linkage class of $\mm^2$.
The latter statement would also be a natural ``next step'' version of
\Cref{main_conjecture},
which implies that all ideals $I$ with tangent space dimension $3d$ are in the  linkage class of $\mm$.
 Unfortunately, this is not always the case, as the following example shows.

 \begin{example}\label{ex:monomialSurplusSixNOTlinkedToKatz} Consider the monomial ideal $I= (x^3,y^3,z^3,yz^2,x^2z,xy^2)$.
A direct calculation shows that $[S/I] \in \Hilb^{14}(\AA^3)$ and $\dim_{\kk} T(I) = 48 =3\cdot 14+6$.
In fact, more is true:
the singularity at $[S/I]$ is smoothly equivalent
to the vertex of a cone over the Grassmannian
$\Gr(2,6) \hookrightarrow \mathbb{P}^{14}$ in its Pl\"ucker embedding.
This can be verified  using 
the Macaulay2 package \texttt{VersalDeformations}
\cite{Ilten}.
We will prove that $I$ is not in the same linkage class as $\mm^2$.
This shows that linkage is a strictly finer equivalence relation than smooth equivalence.

The minimal resolution of $S/I$ is
$$
0 \longrightarrow S(-6)^4 \longrightarrow S(-4)^3 \oplus S(-5)^6  \longrightarrow S(-3)^6 \longrightarrow S \longrightarrow S/I \longrightarrow 0.
$$
Consider the link $J = ((x^3,y^3,z^3):I) = (x^3,y^3,z^3,yz^2,x^2z,xy^2,xyz)$. 
Assume by contradiction that $\mm^2$ is in the same linkage class as $I$.
Since $I$ and $J$ are directly linked, 
then  $\mm^2$ is in the even linkage class of either $I$ or $J$.
We will show that this leads to a contradiction using \cite[Theorem 6.3]{Huneke_Ulrich__Structure_of_linkage}.

Assume that $\mm^2$ is evenly linked to $I$.
Since the graded Betti numbers $\beta_{i,j}(S/I)$ satisfy the inequality 
 $6 = \max{\beta_{3,j}(S/I)} \geq (3-1)  \min{\beta_{1,j}(S/I)} = 6$, \cite[Theorem 6.3]{Huneke_Ulrich__Structure_of_linkage} implies that 
$4 = \dim_\kk(S/\mm^2) \geq  \dim_\kk(S/I)= 14$, contradiction.
The same argument shows that $\mm^2$ is not evenly linked to $J$,
since its minimal resolution is 
$$
0 \longrightarrow S(-6)^3 \longrightarrow S(-4)^6 \oplus S(-5)^3  \longrightarrow S(-3)^7 \longrightarrow S \longrightarrow S/J \longrightarrow 0.
$$

\end{example}

\bibliographystyle{abbrv}
\bibliography{references.bib}

\end{document}